\documentclass[12pt]{amsart}
\usepackage{amsmath,amssymb,amsthm}
\usepackage{mathrsfs}
\usepackage{color}
\usepackage{url}
\usepackage[english]{babel}
\usepackage[T1]{fontenc}
\usepackage{a4wide}
\usepackage[urlcolor=blue,colorlinks=true]{hyperref}
\usepackage{ifthen}
\usepackage{mathtools}
\usepackage{booktabs}
\usepackage{pgfplots}
\pgfplotsset{compat=1.18}
\usepackage{amsmath}
\usepackage{amssymb}

\numberwithin{equation}{section}
\nonstopmode

\theoremstyle{plain}
\newtheorem{theorem}[equation]{Theorem}
\newtheorem{lemma}[equation]{Lemma}
\newtheorem{proposition}[equation]{Proposition}

\theoremstyle{definition}
\newtheorem{definition}[equation]{Definition}
\newtheorem{example}[equation]{Example}
\newtheorem{remark}[equation]{Remark}

\begin{document}

\bibliographystyle{amsplain}

\title[Characterizations of Extremal Hyperbolic Rates]%
{Characterizations of Extremal Hyperbolic Rates via Herglotz Measures and Koenigs Linearization}

\author[R.~Kargar]{Rahim Kargar}
\address{Department of Mathematics and Statistics,
  University of Turku, Turku, Finland}
\email{rahim.r.kargar@utu.fi}


\begin{abstract}
Let $g:\mathbb{D}\to\mathbb{D}$ be a hyperbolic holomorphic self-map of
the unit disk with Denjoy--Wolff point $\tau\in\partial\mathbb{D}$ and
angular derivative $\alpha\in(0,1)$. We say that $g$ has an
\textit{extremal hyperbolic rate} if its forward iterates satisfy the
sharp metric asymptotic
\begin{equation*}
\rho_{\mathbb{D}}(g^{\circ n}(z),w)
=
n\log\frac{1}{\alpha}
+
O(1)
\quad\text{as }n\to\infty,
\end{equation*}
for every $z,w\in\mathbb{D}$. Using the Herglotz--Nevanlinna
representation, we prove that $g$ has an extremal hyperbolic rate if and
only if the associated boundary measure $\sigma$ satisfies
\begin{equation*}
\int_{\partial\mathbb{D}\setminus\{\tau\}}
\log\frac{1}{|\zeta-\tau|}
\,d\sigma(\zeta)
<
\infty.
\end{equation*}
We further show that this condition is equivalent to a non-degenerate
angular asymptotic of the Koenigs linearization: for any conformal map
$\phi_\tau:\mathbb{D}\to\mathbb{H}$ with $\phi_\tau(\tau)=\infty$,
\begin{equation*}
0<
\left|
\angle\lim_{z\to\tau}
\frac{h(z)}{\phi_\tau(z)}
\right|
<
\infty,
\end{equation*}
where $h$ is a Koenigs function.
We then extend the extremal-rate theory beyond the unit disk. For
finitely connected hyperbolic planar domains, the disk characterizations
transfer to ordinary boundary points of the associated deck
transformation group. For holomorphic self-maps of the unit ball
$\mathbb{B}^n$ and for $K$-quasiconformal self-maps, where no comparable
Herglotz representation is available; we establish sufficient boundary
regularity conditions that guarantee the extremal hyperbolic rate. The
quasiconformal result is further extended to finitely connected planar
domains at ordinary boundary points. These results show that, across
the settings considered here, extremal hyperbolic growth is governed by
the non-degeneracy of the boundary linearization at the Denjoy--Wolff
point.
\end{abstract}
\maketitle

\tableofcontents
%%=============================================================
\section{Introduction}
\label{sec:intro}
%%=============================================================

The long-term asymptotic behavior of the iterates of a holomorphic self-map
of the open unit disk
\begin{equation*}
\mathbb D=\{z\in\mathbb C:|z|<1\}
\end{equation*}
is governed by the classical Denjoy--Wolff theorem. This area remains an
active subject of research, with numerous recent developments concerning the iteration theory and asymptotic behavior of holomorphic self-maps; see,
for instance, \cite{Abate, Bak-Pom, bracci2020, Bracci2010, Cowen81,  CDG2025parabolic, CDG2024, Mac, Pom79, Shoikhet} and the references therein. If
$g:\mathbb D\to\mathbb D$ is a holomorphic self-map with no fixed point in
$\mathbb D$, then by the classical Wolff lemma (see, for instance,
\cite{Abate}) there exists a unique point
$\tau\in\partial\mathbb D$, called the Denjoy--Wolff point of $g$, such
that
\begin{equation*}
g^{\circ n}\longrightarrow\tau
\end{equation*}
uniformly in compact subsets of $\mathbb D$. Moreover, by the classical
Julia--Wolff--Carath\'eodory theorem, $\tau$ is a boundary fixed point in
the non-tangential sense, and the angular derivative
\begin{equation*}
\alpha
:=
\angle\lim_{z\to\tau}g'(z)
\end{equation*}
exists with $\alpha\in(0,1]$; see, e.g.,
\cite[Theorem~1.8.4]{bracci2020} or \cite[p.~78]{Shapiro}. Equivalently,
$\alpha$ is the Julia--Carath\'eodory quotient
\begin{equation*}
\alpha
=
\lim_{z\to\tau}\inf
\frac{1-|g(z)|}{1-|z|}.
\end{equation*}
When $\alpha\in(0,1)$, the map $g$ is called \textit{hyperbolic}
(see \cite[Definition~1.8.5]{bracci2020}). In this case, the hyperbolic
distance of every forward orbit from any fixed reference point grows
asymptotically with slope $\log(1/\alpha)$:
\begin{equation*}
\frac{\rho_{\mathbb D}(g^{\circ n}(z),w)}{n}
\longrightarrow
\log\frac1\alpha
\quad\text{as }n\to\infty,
\end{equation*}
for every $z,w\in\mathbb D$ (see Section~\ref{sec:prelim} below).

Although the angular derivative completely determines the first-order
asymptotic growth rate, it does not determine the finer behavior of the
remainder. A natural question is whether the hyperbolic distance admits
the sharper expansion
\begin{equation*}
\rho_{\mathbb D}(g^{\circ n}(z),w)
=
n\log\frac1\alpha+O(1),
\end{equation*}
or whether the remainder may instead be unbounded. More generally, one
may ask which boundary features of the map determine the size of the
deviation from the asymptotic linear profile. In other words, is there a
natural finer invariant that distinguishes hyperbolic self-maps having the
same angular derivative but different asymptotic behavior near their common
Denjoy--Wolff point?

The present paper investigates this finer asymptotic behavior beyond the
first-order linear growth. We say that a hyperbolic self-map $g$ has an
\textit{extremal hyperbolic rate} if
\begin{equation*}
\rho_{\mathbb D}(g^{\circ n}(z),w)
-
n\log\frac1\alpha
\end{equation*}
remains bounded as $n\to\infty$
(Definition~\ref{def:extremal_rate} below). This property distinguishes
hyperbolic self-maps having the same angular derivative but potentially
different asymptotic behavior near their common Denjoy--Wolff point.

This metric problem is also closely related to the theory of composition
operators. For the composition operator
\begin{equation*}
C_g:f\mapsto f\circ g
\end{equation*}
acting on the Hardy space $H^2(\mathbb D)$, the asymptotic behavior of the
iterates of $g$ plays an important role in the spectral, compactness, and
Fredholm properties of $C_g$; see, e.g.,
Elliott and Jury~\cite{elliott}, Moorhouse~\cite{moor}, Cowen and MacCluer
\cite[Chapter~7]{CowenMacCluer}, or
Shapiro~\cite[Chapter~5]{Shapiro}. The extremal metric condition studied
here identifies a natural geometric threshold for the boundedness of the
remainder in the asymptotic orbit expansion.

Our first main result gives a complete characterization of the extremal
hyperbolic rate in terms of the boundary concentration of the
Herglotz--Nevanlinna measure associated with $g$. Specifically, if
$\sigma$ denotes the positive Borel measure arising from the
Herglotz--Nevanlinna representation of $g$, then $g$ has an extremal
hyperbolic rate if and only if
\begin{equation}
\label{eq:intro_logmoment}
\int_{\partial\mathbb D\setminus\{\tau\}}
\log\frac1{|\zeta-\tau|}
\,d\sigma(\zeta)
<
\infty.
\end{equation}
This characterization is obtained by transferring the dynamics to the
upper half-plane via a Cayley transform and adapting the metric criterion
of Cruz-Zamorano and Zarvalis~\cite[Theorem~4.4]{cruz}.

Our second main result establishes an equivalent formulation in terms of
the Koenigs linearizing function. We prove that the extremal hyperbolic
rate holds if and only if the Koenigs function $h$ has a finite and
nonzero angular ratio with a conformal coordinate sending the
Denjoy--Wolff point to infinity, namely,
\begin{equation*}
0<
\left|
\angle\lim_{z\to\tau}
\frac{h(z)}{\phi_\tau(z)}
\right|
<
\infty,
\end{equation*}
where $\phi_\tau:\mathbb D\to\mathbb H$ is any conformal map such that   $\phi_\tau(\tau)=\infty$, and $h$ is a Koenigs function. This provides an intrinsic geometric interpretation of the logarithmic integrability condition
\eqref{eq:intro_logmoment}, relating the boundary concentration of the
Herglotz measure to the boundary behavior of the linearizing coordinate.

The remainder of the paper is organized as follows.
Section~\ref{sec:prelim} reviews the necessary background on hyperbolic
geometry, angular derivatives, and Herglotz--Nevanlinna representations.
A key ingredient is a precise asymptotic expansion for the hyperbolic
distance near the boundary,
\begin{equation*}
\rho_{\mathbb D}(z,w)
=
\log\frac1{1-|z|^2}
+
O(1),
\end{equation*}
together with an explicit sharp constant in the corresponding error
estimate.

Section~\ref{sec:uds} develops the theory in the unit disk. After
transferring the dynamics to the upper half-plane, we establish the
Herglotz characterization of extremality and the equivalent criterion in
terms of the Koenigs function. We also investigate the non-extremal
regime, describing how the concentration of the Herglotz measure near the
Denjoy--Wolff point influences the growth of the metric remainder when
the logarithmic integrability condition fails.

Section~\ref{sec:higher} extends the theory of extremal rate to several
broader settings. For simply connected hyperbolic domains, the Herglotz
and Koenigs characterizations transfer directly through the Riemann
mapping theorem. For hyperbolic domains with finitely many boundary
components, the corresponding equivalence is established at ordinary
boundary points of the associated Fuchsian deck group, while punctures
require a separate treatment. We then consider holomorphic self-maps of
the unit ball $\mathbb B^n$ and $K$-quasiconformal self-maps. In these
higher-dimensional and quasiconformal settings, we obtain sufficient
boundary regularity conditions guaranteeing the extremal linear rate,
rather than complete characterizations of the type available in the
one-dimensional holomorphic setting. Finally, the quasiconformal result
is extended to finitely connected planar domains under the same
restriction to ordinary boundary points of the deck group.\\
\noindent
\textbf{Relation to prior work.}
The iteration theory of holomorphic self-maps of the unit disk and the
upper half-plane is well established; see, for example,
Abate~\cite[Section~1.3]{Abate},
Bracci, Contreras, and D\'iaz-Madrigal
\cite[Section~1.8]{bracci2020},
and Shoikhet~\cite[Chapter~1]{Shoikhet}. The linearization of hyperbolic
self-maps via Schr\"oder's equation \eqref{eq:koenigs_func} originates in
the pioneering work of Koenigs~\cite{K}.

Cruz-Zamorano and Zarvalis~\cite{cruz} proved that extremal convergence
in the upper half-plane is equivalent to the logarithmic integrability
of the associated Nevanlinna measure. The present work adapts this
framework to the unit disk, establishes the corresponding characterization
in terms of the boundary behavior of the Koenigs function, and extends
the metric theory to multiply connected domains and to higher-dimensional
and quasiconformal dynamics.

The relationship between Koenigs functions and composition operators is
a central theme in operator theory; see Cowen and
MacCluer~\cite[Chapter~7]{CowenMacCluer}. The results of this paper
complement that perspective by showing that the boundary behavior of the
linearizing function provides a precise geometric criterion for the
extremal asymptotic behavior of hyperbolic orbits.

%%=============================================================
\section{Background and the extremal rate condition}
\label{sec:prelim}
%%=============================================================
Throughout this paper, $f\asymp g$ means that the ratio
$f/g$ is bounded above and below by positive constants
depending only on fixed parameters, and $f\sim g$ means
$f/g\to 1$ in the relevant limit.
We write $f=O(g)$ if $|f|\le Cg$ for some constant $C>0$
independent of the variable of interest, and $f=o(g)$ if
$f/g\to 0$. When $g=1$ the notation $O(1)$ denotes a
quantity that is bounded independently of the iteration
index $n$, although it may depend on fixed data such
as the map $g$, the base points $z$ and $w$, and the Stolz
parameter $M>1$.
\subsection{Hyperbolic geometry of the disk and the ball}
\label{subsec:hyp_geom}

Let $\mathbb{D}=\{z\in\mathbb{C}:|z|<1\}$ be the unit disk and
$\mathbb{H}=\{z\in\mathbb{C}:\operatorname{Im}z>0\}$ be the
upper half-plane. Set $\overline{\mathbb{D}}=\mathbb{D}\cup \partial \mathbb{D}$. Also, let
$\mathbb{B}^n=\{z\in\mathbb{C}^n:|z|<1\}$ be
the unit ball in $\mathbb{C}^n$.
We use the hyperbolic metric normalized to have a Gaussian
curvature $-1$, so that its density in $\mathbb{D}$ is
\begin{equation*}
\label{lambda_density}
\lambda_{\mathbb{D}}(z)\,|dz|
=\frac{2\,|dz|}{1-|z|^2},
\end{equation*}
and the hyperbolic distance between $z,w\in\mathbb{D}$ is
\begin{equation}
\label{eq:hyp_dist}
\rho_{\mathbb{D}}(z,w)
=2\tanh^{-1} \left|\frac{z-w}{1-\bar w z}\right|
=\log\frac{1+|\phi_w(z)|}{1-|\phi_w(z)|},
\end{equation}
where
\begin{equation}
\label{phi}
\phi_w(z)=\frac{z-w}{1-\bar w z}
\end{equation}
is the M\"obius automorphism of $\mathbb{D}$ mapping $w$
to $0$; see~\cite[Theorem~2.2]{BM}. The quantity $|\phi_w(z)|$ is called the pseudo--hyperbolic distance.

In the unit ball $\mathbb{B}^n$ we use the same curvature
normalization $-1$ such that
\begin{equation*}
\rho_{\mathbb{B}^n}(z,w)
=\inf_\gamma\int_\gamma\frac{2\,|d\xi|}{1-|\xi|^2},
\end{equation*}
where the infimum is taken over all rectifiable curves $\gamma$
joining $z$ to $w$ in $\mathbb{B}^n$.
In dimension $n=1$, this agrees with~\eqref{eq:hyp_dist}.
The hyperbolic metric $\rho_{\mathbb{B}^n}$ of the unit ball $\mathbb{B}^n$ can be expressed as follows (see, e.g., \cite[p. 40]{B83} or \cite[p. 55]{HKV}):
\begin{equation}
\label{hyperb-unit ball}
\sinh \frac{\rho_{\mathbb{B}^n}(z,w)}{2}
=\frac{|z-w|}{\sqrt{1-|z|^2}\sqrt{1-|w|^2}}.
\end{equation}

The following lemma is the key asymptotic estimate that
relates the hyperbolic distance to the Euclidean boundary distance.
We include a self-contained proof because the precise
uniformity in the Stolz regions is used repeatedly in what
follows.

\begin{lemma} \label{lem:boundary_asymp}
Let $w\in\mathbb{D}$ be fixed. As $z\to\tau$ for any $\tau\in\partial\mathbb{D}$, we have
\begin{equation} \label{eq:boundary_asymp}
\rho_{\mathbb{D}}(z,w) =\log\frac{1}{1-|z|^2}+O(1),
\end{equation}
where the $O(1)$ term is uniform in $\mathbb{D}$ with a bound depending only on $w$. More precisely,
\begin{equation} \label{eq:boundary_asymp_precise}
\left|\rho_{\mathbb{D}}(z,w) -\log\frac{1}{1-|z|^2}\right| \le\log\frac{4(1+|w|)}{1-|w|}
\end{equation}
for all $z\in\mathbb{D}$. The inequality is sharp.
\end{lemma}

\begin{proof}
Fix $w\in\mathbb{D}$ and let $\phi_w$ be defined as in~\eqref{phi}.
We begin with the standard identity for the pseudo--hyperbolic distance
\begin{equation}\label{identity-1-phiwz}
1-|\phi_w(z)|^2 = \frac{(1-|z|^2)(1-|w|^2)}{|1-\bar w z|^2}.
\end{equation}
By factoring the left--hand side, we can isolate $1-|\phi_w(z)|$
\begin{equation}\label{eq:one_minus_phi}
1-|\phi_w(z)| = \frac{1-|\phi_w(z)|^2}{1+|\phi_w(z)|} = \frac{(1-|z|^2)(1-|w|^2)}{|1-\bar w z|^2 (1+|\phi_w(z)|)}.
\end{equation}
Taking the reciprocal logarithm of both sides of~\eqref{eq:one_minus_phi} gives
\begin{equation}\label{eq:neg_log_phi}
-\log(1-|\phi_w(z)|) = \log\frac{1}{1-|z|^2} + \log\frac{|1-\bar w z|^2}{1-|w|^2} + \log(1+|\phi_w(z)|).
\end{equation}
From the definition of the hyperbolic metric in~\eqref{eq:hyp_dist}, we have
\begin{equation*}
\rho_{\mathbb{D}}(z,w) =  \log(1+|\phi_w(z)|)-\log(1-|\phi_w(z)|).
\end{equation*}
Substituting the expression from~\eqref{eq:neg_log_phi} into this formula yields the exact identity of the error term
\begin{equation}\label{eq:exact_error}
\rho_{\mathbb{D}}(z,w) - \log\frac{1}{1-|z|^2} = \log\frac{|1-\bar w z|^2}{1-|w|^2} + 2\log(1+|\phi_w(z)|).
\end{equation}
To find a uniform upper bound for the absolute value of this difference, we apply standard geometric estimates in the unit disk. For all $z, w \in \mathbb{D}$, the triangle inequality gives:
\begin{equation*}
1-|w| \le |1-\bar w z| \le 1+|w|.
\end{equation*}
Furthermore, since $|\phi_w(z)| < 1$, the expression $1+|\phi_w(z)|$ is strictly bounded between $1$ and $2$. Applying these bounds to the two components of the error term in~\eqref{eq:exact_error} separately, we obtain
\begin{equation*}
\log\frac{(1-|w|)^2}{1-|w|^2} + 2\log(1) \le \log\frac{|1-\bar w z|^2}{1-|w|^2} + 2\log(1+|\phi_w(z)|) \le \log\frac{(1+|w|)^2}{1-|w|^2} + 2\log(2).
\end{equation*}
Simplifying the fractions inside the logarithms gives
\begin{equation*}
\log\frac{1-|w|}{1+|w|} \le \rho_{\mathbb{D}}(z,w) - \log\frac{1}{1-|z|^2} \le \log\frac{1+|w|}{1-|w|} + \log 4 = \log\frac{4(1+|w|)}{1-|w|}.
\end{equation*}
Since $\log\frac{1-|w|}{1+|w|} < 0$, the lower bound in absolute value equals $\log\frac{1+|w|}{1-|w|}$, which is strictly smaller than the upper bound $\log\frac{4(1+|w|)}{1-|w|}$. Therefore, both tails of the two-sided estimate are controlled by the upper bound, and we secure the uniform estimate:
\begin{equation*}
\left|\rho_{\mathbb{D}}(z,w) - \log\frac{1}{1-|z|^2}\right| \le \log\frac{4(1+|w|)}{1-|w|}
\end{equation*}
for all $z \in \mathbb{D}$. This directly delivers~\eqref{eq:boundary_asymp_precise}.

Finally, since this bound holds for every $z\in\mathbb{D}$ and depends only on $w$, it holds in particular as $z\to\tau$ for any $\tau\in\partial\mathbb{D}$ approached in any manner, which yields the uniform asymptotic relation~\eqref{eq:boundary_asymp}.

We now show that the uniform bound in~\eqref{eq:boundary_asymp_precise} is exactly sharp. It is attained as a limiting value, and the exact asymptotic behavior of the error can be computed explicitly.

We first note that for $w=0$, the formula~\eqref{eq:hyp_dist} gives
\begin{equation*}
\rho_{\mathbb{D}}(z,0) =\log \frac{1+|z|}{1-|z|}= \log\frac{1}{1-|z|^2} + 2\log(1+|z|),
\end{equation*}
so the error equals $2\log(1+|z|) \to \log 4$ as $|z|\to 1^-$.

More generally, we prove that for any $w\in\mathbb{D}$ and $\tau\in\partial\mathbb{D}$, the limit of the error as $z\to\tau$ is given by
\begin{equation}
\label{eq:exact_limit}
\lim_{z\to\tau}
  \left(\rho_{\mathbb{D}}(z,w) - \log\frac{1}{1-|z|^2}\right)
= \log\frac{4|1-\bar{w}\tau|^2}{1-|w|^2}.
\end{equation}
Since $|1-\bar w\tau|^2$ attains its maximum value $(1+|w|)^2$ over $\tau\in\partial\mathbb{D}$ precisely at the point $\tau$, where $\bar w\tau=-|w|$, taking the supremum of~\eqref{eq:exact_limit} over $\tau\in\partial\mathbb{D}$ gives
\begin{equation*}
\sup_{\tau\in\partial\mathbb{D}}\log\frac{4|1-\bar w\tau|^2}{1-|w|^2} = \log\frac{4(1+|w|)^2}{1-|w|^2} = \log\frac{4(1+|w|)}{1-|w|},
\end{equation*}
which coincides exactly with the constant in~\eqref{eq:boundary_asymp_precise}. Hence, the bound in Lemma~\ref{lem:boundary_asymp} is not merely an $O(1)$ estimate up to an unspecified constant, but is attained in the limit along the boundary direction opposite to $w/|w|$.

\noindent
It remains to prove \eqref{eq:exact_limit}. From the definition of the hyperbolic distance, we have
\begin{equation*}
\rho_{\mathbb{D}}(z,w) = \log\frac{1+|\phi_w(z)|}{1-|\phi_w(z)|} = \log\frac{(1+|\phi_w(z)|)^2}{1-|\phi_w(z)|^2}.
\end{equation*}
Expanding the logarithm yields
\begin{equation*}
\rho_{\mathbb{D}}(z,w) = \log\frac{1}{1-|\phi_w(z)|^2} + 2\log(1+|\phi_w(z)|).
\end{equation*}
Substituting the identity~\eqref{identity-1-phiwz} into this equation gives
\begin{equation*}
\rho_{\mathbb{D}}(z,w) = \log\frac{|1-\bar{w}z|^2}{(1-|z|^2)(1-|w|^2)} + 2\log(1+|\phi_w(z)|).
\end{equation*}
Subtracting $\log\frac{1}{1-|z|^2}$ from both sides, we isolate the exact error expression
\begin{align*}
\rho_{\mathbb{D}}(z,w) - \log\frac{1}{1-|z|^2}
&= \log\frac{|1-\bar{w}z|^2}{1-|w|^2} + 2\log(1+|\phi_w(z)|).
\end{align*}
As $z\to\tau$, we have $|\phi_w(z)|\to 1$. This holds regardless of the approach mode, since $\phi_w$ extends continuously to $\overline{\mathbb{D}}$ (its unique pole $1/\bar w$ lies outside $\overline{\mathbb{D}}$) and maps $\partial\mathbb{D}$ onto $\partial\mathbb{D}$. Thus, $2\log(1+|\phi_w(z)|)\to 2\log 2 = \log 4$. Furthermore, by continuity, $|1-\bar{w}z|^2 \to |1-\bar{w}\tau|^2$.

Taking the limit yields the following
\begin{equation*}
\rho_{\mathbb{D}}(z,w) - \log\frac{1}{1-|z|^2}
\;\longrightarrow\;
\log\frac{|1-\bar{w}\tau|^2}{1-|w|^2} + \log 4
= \log\frac{4|1-\bar{w}\tau|^2}{1-|w|^2},
\end{equation*}
which completes the proof of~\eqref{eq:exact_limit}. The proof is now complete.
\end{proof}

Figure~\ref{fig:boundary_asymptotic} illustrates this phenomenon in the case $w=0$. Both
$\rho_{\mathbb D}(z,0)$ and $\log\frac{1}{1-|z|^2}$ diverge as $|z|\to1^-$, yet their difference
remains bounded and increases monotonically to the sharp constant $\log 4$, confirming that the
bound in \eqref{eq:boundary_asymp_precise} is approached but never attained.
%-----------------------------------------------------------------------------------------------------------------------------------------
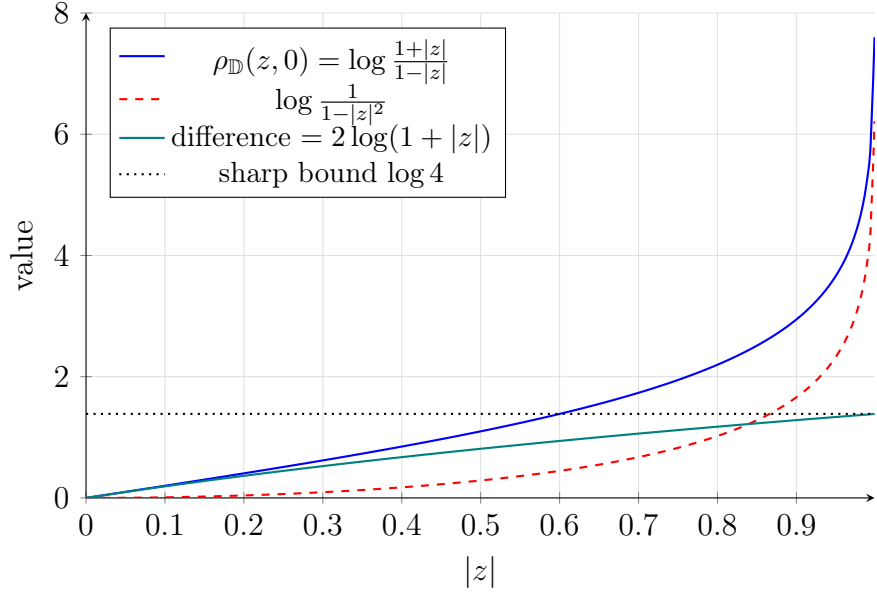
\begin{figure}
    \centering

\begin{tikzpicture}
\begin{axis}[
    width=12cm, height=8cm,
    xlabel={$|z|$},
    ylabel={value},
    xmin=0, xmax=0.999,
    ymin=0, ymax=8,
    legend pos=north west,
    legend style={font=\small},
    domain=0:0.999,
    samples=200,
    smooth,
    axis lines=left,
    grid=both,
    grid style={line width=.1pt, draw=gray!25},
]

% rho_D(z,0) = log((1+r)/(1-r))
\addplot[blue, thick] {ln((1+x)/(1-x))};
\addlegendentry{$\rho_{\mathbb D}(z,0)=\log\frac{1+|z|}{1-|z|}$}

% log(1/(1-r^2))
\addplot[red, thick, dashed] {ln(1/(1-x^2))};
\addlegendentry{$\log\frac{1}{1-|z|^2}$}

% difference = 2 log(1+r)
\addplot[teal, thick] {2*ln(1+x)};
\addlegendentry{difference $=2\log(1+|z|)$}

% sharp asymptotic bound log 4
\addplot[black, dotted, thick, domain=0:0.999] {ln(4)};
\addlegendentry{sharp bound $\log 4$}

\end{axis}
\end{tikzpicture}
    \caption{Comparison of $\rho_{\mathbb D}(z,0)=\log\frac{1+|z|}{1-|z|}$ (solid blue) and its
leading-order approximation $\log\frac{1}{1-|z|^2}$ (dashed red) as $|z|\to 1^-$. Both quantities
diverge at the boundary, but their difference (solid teal) is given exactly, for $w=0$, by
$2\log(1+|z|)$, and increases monotonically to the sharp bound $\log 4$ (dotted) predicted by
Lemma~\ref{lem:boundary_asymp} with $w=0$; the bound is approached as $|z|\to 1^-$ but never
attained for $z\in\mathbb D$.}
\label{fig:boundary_asymptotic}
\label{fig:placeholder}
\end{figure}
%-----------------------------------------------------------------------------------------------------------------------------------------
\subsection{Denjoy--Wolff theory and angular derivatives}
\label{subsec:dw}

Let $g:\mathbb{D}\to\mathbb{D}$ be holomorphic.
For $z\in\mathbb{D}$, the orbit of $z$ under $g$ or the iterations of $g$
is the sequence $z, g(z), g^{\circ 2}(z),\ldots$,
where $g^{\circ 0}(z)=z$ and
$g^{\circ(n+1)}(z)=g(g^{\circ n}(z))$ for $n\in\mathbb{N}:=\{1,2,\ldots\}$.
If $g$ has no fixed point in $\mathbb{D}$, the Denjoy--Wolff
theorem~\cite[Theorem~1.8.4]{bracci2020} asserts that there
exists a unique point $\tau\in\partial\mathbb{D}$,
called Denjoy--Wolff point of $g$ such that
$g^{\circ n}(z)\to\tau$ for every $z\in\mathbb{D}$.

We call $g$ \textit{hyperbolic} if its Denjoy--Wolff point
lies in $\partial\mathbb{D}$ and the non-tangential (or angular) limit
\begin{equation}
\label{eq:multiplier_def}
\alpha:=\alpha_\tau(g)
=\angle\lim_{z\to\tau}\inf\frac{1-|g(z)|}{1-|z|}
\end{equation}
exists and satisfies $\alpha\in(0,1)$.
The quantity $\alpha$ is called the \textit{multiplier} of
$g$ or the boundary dilatation coefficient
at $\tau$.
By the Julia--Wolff--Carath\'eodory theorem
\cite[Theorem~1.7.3]{bracci2020}, the existence of the
limit~\eqref{eq:multiplier_def} with $\alpha<\infty$
implies that the angular derivative $g'(\tau)$ exists, is
real and positive, and satisfies $g'(\tau)=\alpha$.

\subsection{The extremal and non--extremal rate conditions}
\label{subsec:extremal_def}

We first motivate the definition of an extremal rate. Recall that a
Stolz region at $\tau\in\partial\mathbb{D}$ with parameter $M>1$ is
defined by
\begin{equation*}
S(\tau,M)
=
\{z\in\mathbb{D}: |z-\tau|<M(1-|z|)\}.
\end{equation*}
The condition $M>1$ is essential. Indeed, by the reverse triangle
inequality,
\begin{equation*}
|z-\tau|\geq \bigl||\tau|-|z|\bigr|=1-|z|,
\end{equation*}
so $S(\tau,M)=\varnothing$ when $M\leq1$.

A sequence $\{z_n\}\subset\mathbb{D}$ is said to converge
\emph{non--tangentially} to $\tau\in\partial\mathbb{D}$ if
$z_n\to\tau$ and, for some $M>1$, the sequence is eventually contained
in $S(\tau,M)$; that is, there exists $N\in\mathbb{N}$ such that
\begin{equation*}
|z_n-\tau|<M(1-|z_n|),\quad n\geq N.
\end{equation*}
Since finitely many terms are irrelevant for asymptotic statements,
one may, after discarding finitely many terms, assume that this
inequality holds for all $n$.

Assume that $g:\mathbb{D}\to\mathbb{D}$ is a hyperbolic self-map with
Denjoy--Wolff point $\tau\in\partial\mathbb{D}$, and let
$\alpha\in(0,1)$ denote its angular derivative at $\tau$. By the
Julia--Wolff--Carath\'eodory theorem,
\begin{equation*}
\label{eq:jwc_radial_ratio}
\frac{1-|g(z)|}{1-|z|}
\longrightarrow \alpha,
\quad
(z\to\tau\ \text{non--tangentially}).
\end{equation*}
Moreover, the Julia--Wolff--Carath\'eodory theorem implies that
non--tangential convergence is preserved under $g$, that is, if
$z\to\tau$ non--tangentially, then
\begin{equation*}
g(z)\to\tau
\quad\text{(non--tangentially)}.
\end{equation*}
Consequently,
\begin{equation*}
\frac{1-|g(z)|^2}{\alpha(1-|z|^2)}
=
\frac{1-|g(z)|}{\alpha(1-|z|)}
\frac{1+|g(z)|}{1+|z|}
\longrightarrow1.
\end{equation*}
Thus,
\begin{equation*}
\label{eq:boundary_ratio_squared}
1-|g(z)|^2
=
\alpha(1-|z|^2)(1+o(1))
\end{equation*}
as $z\to\tau$ non--tangentially. In particular,
\begin{equation*}
\label{eq:log_boundary_ratio}
\log\frac{1-|z|^2}{1-|g(z)|^2}
\longrightarrow
\log\frac{1}{\alpha}.
\end{equation*}

To compare the corresponding hyperbolic distances, we must use the
more precise form of Lemma~\ref{lem:boundary_asymp}, rather than merely
its $O(1)$ estimate. For fixed $w\in\mathbb{D}$, set
\begin{equation*}
E(y,w)
:=
\rho_{\mathbb{D}}(y,w)
-
\log\frac{1}{1-|y|^2}.
\end{equation*}
By the explicit formula for the hyperbolic distance used in the proof
of Lemma~\ref{lem:boundary_asymp}, the function
$E(\,\cdot\,,w)$ extends continuously to $\overline{\mathbb{D}}$.
In particular, if $y\to\tau$, then
\begin{equation*}
E(y,w)\longrightarrow E(\tau,w).
\end{equation*}
Since both $z$ and $g(z)$ converge to $\tau$ non--tangentially, we have
\begin{equation*}
E(g(z),w)-E(z,w)\longrightarrow0.
\end{equation*}
Therefore,
\begin{align*}
\rho_{\mathbb{D}}(g(z),w)-\rho_{\mathbb{D}}(z,w)
&=
\log\frac{1-|z|^2}{1-|g(z)|^2}
+
E(g(z),w)-E(z,w) \notag\\
&=
\log\frac{1}{\alpha}+o(1)
\label{eq:single_step_rate}
\end{align*}
as $z\to\tau$ non--tangentially. Equivalently,
\begin{equation}
\label{eq:single_step_rate_equiv}
\rho_{\mathbb{D}}(g(z),w)
=
\rho_{\mathbb{D}}(z,w)
+
\log\frac{1}{\alpha}
+
\varepsilon(z),
\end{equation}
where
\begin{equation*}
\varepsilon(z)\longrightarrow0,
\quad
(z\to\tau\ \text{non--tangentially}).
\end{equation*}

Now let
\begin{equation*}
z_n=g^{\circ n}(z).
\end{equation*}
Since $g$ is hyperbolic, $z_n\to\tau$ non--tangentially. Hence, for
each fixed initial point $z\in\mathbb{D}$, there exists $N=N(z)$ such
that the asymptotic relation \eqref{eq:single_step_rate_equiv} applies
to $z_k$ for all $k\geq N$. Summing the relation from $N$ to $n-1$ gives
\begin{equation}
\label{eq:telescoping_rate}
\rho_{\mathbb{D}}(z_n,w)
=
\rho_{\mathbb{D}}(z_N,w)
+
(n-N)\log\frac{1}{\alpha}
+
\sum_{k=N}^{n-1}\varepsilon(z_k).
\end{equation}
Since $\varepsilon(z_k)\to0$ as $k\to\infty$, Ces\`aro's theorem implies
\begin{equation*}
\frac{1}{n}
\sum_{k=N}^{n-1}\varepsilon(z_k)
\longrightarrow 0.
\end{equation*}
Consequently,
\begin{equation}
\label{eq:asymptotic_hyperbolic_rate}
\lim_{n\to\infty}
\frac{\rho_{\mathbb{D}}(z_n,w)}{n}
=
\log\frac{1}{\alpha}.
\end{equation}
Thus, the multiplier $\alpha$ determines the asymptotic linear
hyperbolic rate of the orbit. The remaining issue is to determine
whether the accumulated error term in \eqref{eq:telescoping_rate}
remains bounded or exhibits a systematic one-sided divergence. This
requires an additional estimate, which we establish next using
Julia's lemma. Indeed, we find a universal lower bound for $\rho_{\mathbb{D}}(g^{\circ n}(z),w)$.

%-------------------------------------------------------------------------------------------------------------------------------------
\begin{proposition}\label{prop:julia-lower-bound}
Let $g:\mathbb{D}\to\mathbb{D}$ be hyperbolic with Denjoy--Wolff point
$\tau\in\partial\mathbb{D}$ and multiplier $\alpha\in(0,1)$. For every
$z,w\in\mathbb{D}$,
\begin{equation}
\label{eq:julia-lower-bound}
\rho_{\mathbb{D}}(g^{\circ n}(z),w)
\ge
n\log\frac{1}{\alpha}
-
C(z,w;\tau),
\quad n\in\mathbb{N}\cup\{0\},
\end{equation}
where
\begin{equation*}
C(z,w;\tau)
=
\log\left(
16R_\tau(z)\frac{1+|w|}{1-|w|}
\right),
\quad
R_\tau(y):=
\frac{|\tau-y|^2}{1-|y|^2}.
\end{equation*}
\end{proposition}

\begin{proof}
Since $\tau$ is the Denjoy--Wolff point of the hyperbolic map $g$,
its boundary dilation coefficient at $\tau$ equals the multiplier
$\alpha$. Julia's lemma therefore gives
\begin{equation*}
R_\tau(g(y))\leq \alpha R_\tau(y),
\quad y\in\mathbb{D}.
\end{equation*}
Iterating this inequality along the orbit
$z_n=g^{\circ n}(z)$ yields
\begin{equation*}
R_\tau(z_n)\leq \alpha^n R_\tau(z),
\quad n\geq0.
\end{equation*}
Since $|\tau-z_n|\geq1-|z_n|$, we obtain
\begin{equation*}
(1-|z_n|)^2
\leq |\tau-z_n|^2
=R_\tau(z_n)(1-|z_n|^2)
\leq \alpha^nR_\tau(z)(1-|z_n|)(1+|z_n|).
\end{equation*}
Dividing by $1-|z_n|>0$ and using $1+|z_n|<2$ gives
\begin{equation*}
1-|z_n|<2\alpha^nR_\tau(z),
\end{equation*}
and hence,
\begin{equation*}
1-|z_n|^2
=
(1-|z_n|)(1+|z_n|)
<
4\alpha^nR_\tau(z).
\end{equation*}
Therefore,
\begin{equation*}
\log\frac{1}{1-|z_n|^2}
>
n\log\frac{1}{\alpha}
-
\log\bigl(4R_\tau(z)\bigr).
\end{equation*}
By the global estimate in
Lemma~\ref{lem:boundary_asymp}, valid for every $y\in\mathbb{D}$,
\begin{equation*}
\rho_{\mathbb{D}}(y,w)
\geq
\log\frac{1}{1-|y|^2}
-
\log\frac{4(1+|w|)}{1-|w|}.
\end{equation*}
Applying this with $y=z_n$ gives
\begin{align*}
\rho_{\mathbb{D}}(z_n,w)
&>
n\log\frac{1}{\alpha}
-\log\bigl(4R_\tau(z)\bigr)
-\log\frac{4(1+|w|)}{1-|w|}\\
&=
n\log\frac{1}{\alpha}
-
\log\left(
16R_\tau(z)\frac{1+|w|}{1-|w|}
\right).
\end{align*}
Thus, \eqref{eq:julia-lower-bound} holds for $n\geq1$.

For $n=0$, Lemma~\ref{lem:boundary_asymp} gives
\begin{equation*}
\rho_{\mathbb{D}}(z,w)
\geq
\log\frac{1}{1-|z|^2}
-
\log\frac{4(1+|w|)}{1-|w|}.
\end{equation*}
Moreover,
\begin{equation*}
R_\tau(z)
=
\frac{|\tau-z|^2}{1-|z|^2}
\geq
\frac{(1-|z|)^2}{1-|z|^2}
=
\frac{1-|z|}{1+|z|},
\end{equation*}
and therefore
\begin{equation*}
\frac{4R_\tau(z)}{1-|z|^2}
=
\frac{4|\tau-z|^2}{(1-|z|^2)^2}
\geq
\frac{4}{(1+|z|)^2}
>1.
\end{equation*}
Hence, the same lower bound holds for $n=0$.
\end{proof}

\begin{remark}
The above proof gives a completely explicit lower-bound constant depending
only on the initial point $z$, the Denjoy--Wolff point $\tau$, and the
base point $w$. In particular, no non--tangential convergence of the
orbit is required. The estimate follows globally from Julia's lemma,
the elementary inequality
\begin{equation*}
|\tau-y|\geq1-|y|,
\end{equation*}
and the global form of Lemma~\ref{lem:boundary_asymp}.
Consequently, the lower bound holds for every $n\geq0$.
\end{remark}

The universal lower bound in Proposition~\ref{prop:julia-lower-bound}
shows that, once the orbit has entered the region where the
single--step asymptotic \eqref{eq:single_step_rate_equiv} applies, the
accumulated error in the telescoping formula is bounded from below.
It is not, however, forced to be bounded above: a sequence of terms
tending to zero may have partial sums that diverge to $+\infty$.
Thus, the relevant distinction is whether the deviation from the
linear term $n\log(1/\alpha)$ remains bounded or becomes unbounded
above. In the former case, we obtain the extremal rate defined in
Definition~\ref{def:extremal_rate} below; in the latter case, the rate is
non--extremal. Notice that non--extremality does not mean a larger
asymptotic linear slope: by \eqref{eq:asymptotic_hyperbolic_rate}, the
slope is still $\log(1/\alpha)$. Rather, it means that the deviation
from this linear model is unbounded. This distinction is the central object of study below.

Next, we show that the slope $\log(1/\alpha)$ is optimal.
\begin{proposition}
\label{prop:sharpness}
The coefficient $\log(1/\alpha)$ in
Proposition~\ref{prop:julia-lower-bound} is optimal: it cannot be
replaced by any larger constant. More precisely, for every
$\alpha\in(0,1)$ there exist a hyperbolic automorphism
$g:\mathbb{D}\to\mathbb{D}$ with Denjoy--Wolff point $\tau=1$ and
multiplier $\alpha$, and points $z,w\in\mathbb{D}$, such that
\begin{equation}
\label{eq:sharp_exact_rate}
\rho_{\mathbb{D}}(g^{\circ n}(z),w)
=
n\log\frac{1}{\alpha}
\quad\text{for every }n\geq0.
\end{equation}
Thus, the lower bound in Proposition~\ref{prop:julia-lower-bound}
is sharp at the level of the linear coefficient, and in this example
the additive error vanishes identically.

In particular, if $\beta<\alpha$, then there is no constant $C$,
independent of the map and the orbit, such that
\begin{equation*}
\rho_{\mathbb{D}}(g^{\circ n}(z),w)
\geq
n\log\frac{1}{\beta}-C
\end{equation*}
holds for all $n\geq0$ throughout the class of hyperbolic
self--maps with multiplier $\alpha$.
\end{proposition}

\begin{proof}
Fix $\alpha\in(0,1)$ and let $\tau=1$. Consider hyperbolic
automorphism $g:\mathbb{D}\to \mathbb{D}$
\begin{equation}
\label{eq:sharpness_automorphism}
g(z)
=
\frac{(1+\alpha)z+(1-\alpha)}
     {(1-\alpha)z+(1+\alpha)}.
\end{equation}
It fixes $\pm1$, and
\begin{equation*}
g'(z)
=
\frac{4\alpha}
{\bigl((1-\alpha)z+(1+\alpha)\bigr)^2}.
\end{equation*}
Consequently,
\begin{equation*}
g'(1)=\alpha,
\quad
g'(-1)=\frac{1}{\alpha}.
\end{equation*}
Thus, $g$ is hyperbolic, its Denjoy--Wolff point is $\tau=1$, and its
multiplier at $\tau$ is $\alpha$.

We show that the map $g$ preserves the real diameter $(-1,1)$. Define
\begin{equation*}
\psi:(-1,1)\to\mathbb{R},
\quad
\psi(x)=\log\frac{1+x}{1-x}.
\end{equation*}
Since
\begin{equation*}
\psi'(x)=\frac{2}{1-x^2},
\end{equation*}
which is precisely the density of the hyperbolic metric
$\lambda_{\mathbb{D}}$ restricted to the real axis, and since
$(-1,1)$ is a hyperbolic geodesic of $\mathbb{D}$, we have
\begin{equation*}
\rho_{\mathbb{D}}(x,y)
=
|\psi(x)-\psi(y)|,
\quad (x,y\in(-1,1)).
\end{equation*}
Thus, $\psi$ is an isometry from
$((-1,1),\rho_{\mathbb{D}})$ onto $(\mathbb{R},|\cdot|)$.

Since $g$ is an automorphism of $\mathbb{D}$, it is an isometry of
$\rho_{\mathbb{D}}$. Moreover, it maps $(-1,1)$ onto itself and fixes
both endpoints. Hence,
\begin{equation*}
T:=\psi\circ g\circ\psi^{-1}
\end{equation*}
is an orientation--preserving isometry of $\mathbb{R}$ and therefore
has the form
\begin{equation*}
T(t)=t+c
\end{equation*}
for some $c\in\mathbb{R}$. Evaluating at $t=0$ gives
\begin{equation*}
c
=T(0)
=\psi(g(0))
=\psi\left(\frac{1-\alpha}{1+\alpha}\right)
=
\log
\frac{
1+\frac{1-\alpha}{1+\alpha}
}{
1-\frac{1-\alpha}{1+\alpha}
}
=
\log\frac{1}{\alpha}.
\end{equation*}
Therefore
\begin{equation*}
T(t)=t+\log\frac{1}{\alpha},
\end{equation*}
and hence,
\begin{equation*}
\psi(g^{\circ n}(0))
=
T^{\circ n}(0)
=
n\log\frac{1}{\alpha},
\quad(n\geq0).
\end{equation*}
Taking $z=w=0$ and using $\psi(0)=0$, we obtain
\begin{equation*}
\rho_{\mathbb{D}}(g^{\circ n}(0),0)
=
\left|\psi(g^{\circ n}(0))\right|
=
n\log\frac{1}{\alpha},
\end{equation*}
which proves \eqref{eq:sharp_exact_rate}.

Finally, let $\beta<\alpha$. Then
\begin{equation*}
\log\frac{1}{\beta}
>
\log\frac{1}{\alpha}.
\end{equation*}
If there is a constant $C$ independent of the map and the orbit
such that
\begin{equation*}
\rho_{\mathbb{D}}(g^{\circ n}(z),w)
\geq
n\log\frac{1}{\beta}-C
\end{equation*}
for all hyperbolic self--maps with multiplier $\alpha$, applying this
to the automorphism \eqref{eq:sharpness_automorphism} with $z=w=0$ gives
\begin{equation*}
n\log\frac{1}{\alpha}
\geq
n\log\frac{1}{\beta}-C,
\quad(n\geq0).
\end{equation*}
Equivalently,
\begin{equation*}
n\left(
\log\frac{1}{\beta}
-
\log\frac{1}{\alpha}
\right)
\leq C,
\end{equation*}
which is impossible as $n\to\infty$. Hence, a larger coefficient cannot
hold uniformly in the class, proving the sharpness of
$\log(1/\alpha)$.
\end{proof}

\begin{remark}
\label{rem:constant-not-sharp}
We note that the constant in the universal lower bound is not optimal.
Sharpness of the rate does not imply optimality of the particular
constant
\begin{equation*}
C(z,w;\tau)
=
\log\left(
16R_\tau(z)\frac{1+|w|}{1-|w|}
\right)
\end{equation*}
obtained in Proposition~\ref{prop:julia-lower-bound}. Indeed, for the
automorphism in Proposition~\ref{prop:sharpness}, with
$\tau=1$ and $z=w=0$, we have
\begin{equation*}
R_\tau(0)=\frac{|1-0|^2}{1-|0|^2}=1,
\end{equation*}
and hence,
\begin{equation*}
C(0,0;\tau)
=
\log16.
\end{equation*}
On the other hand, Proposition~\ref{prop:sharpness} gives
\begin{equation*}
\rho_{\mathbb{D}}(g^{\circ n}(0),0)
-
n\log\frac{1}{\alpha}
=
0
\quad\text{for every }n\geq0.
\end{equation*}
Thus, the actual error is identically zero, while the universal
constant supplied by Proposition~\ref{prop:julia-lower-bound} is
$\log16$.

The lack of sharpness comes from the estimates used in the proof,
in particular the replacement of $1+|z_n|$ by the crude upper bound
$2$. The estimate $1+|z_n|\to2$ alone does not, however, yield a
uniformly sharper constant of the same form. Determining an optimal
pointwise constant in the universal lower bound is a separate
question and is not pursued here.
\end{remark}

We are now ready to define the extremal hyperbolic rate.

\begin{definition}
\label{def:extremal_rate}
Let $g:\mathbb{D}\to\mathbb{D}$ be hyperbolic with Denjoy--Wolff point
$\tau\in\partial\mathbb{D}$ and boundary multiplier
$\alpha\in(0,1)$. We say that $g$ has an
\emph{extremal hyperbolic rate} if, for some (equivalently, every)
$w\in\mathbb{D}$ and some (equivalently, every) $z\in\mathbb{D}$,
\begin{equation}
\label{eq:extremal_rate}
\rho_{\mathbb{D}}(g^{\circ n}(z),w)
=
n\log\frac{1}{\alpha}+O(1),
\quad (n\to\infty).
\end{equation}
Equivalently, for some (equivalently, every) $z,w\in\mathbb{D}$,
there exist constants $C=C(z,w,g)>0$ and
$n_0=n_0(z,w,g)\in\mathbb{N}$ such that
\begin{equation*}
\left|
\rho_{\mathbb{D}}(g^{\circ n}(z),w)
-
n\log\frac{1}{\alpha}
\right|
\le C
\quad\text{for all }n\ge n_0.
\end{equation*}

By Proposition~\ref{prop:julia-lower-bound}, the quantity
\begin{equation*}
\rho_{\mathbb{D}}(g^{\circ n}(z),w)
-
n\log\frac{1}{\alpha}
\end{equation*}
is bounded below (for fixed $z,w$). Hence, $g$ has an extremal
hyperbolic rate if and only if this quantity is also bounded above.
Thus, extremality means that the orbit has only a bounded deviation
from the linear growth $n\log(1/\alpha)$. In particular, it is not a
distinction in the first-order asymptotic slope: all hyperbolic maps
under consideration have the same slope,
\begin{equation*}
\lim_{n\to\infty}
\frac{\rho_{\mathbb{D}}(g^{\circ n}(z),w)}{n}
=
\log\frac{1}{\alpha},
\end{equation*}
whereas in the extremal case the deviation from this linear term
remains bounded.

The property in \eqref{eq:extremal_rate} is independent of the choice
of the base point $w$. Indeed, for $w_1,w_2\in\mathbb{D}$, the triangle
inequality gives
\begin{equation*}
\left|
\rho_{\mathbb{D}}(g^{\circ n}(z),w_1)
-
\rho_{\mathbb{D}}(g^{\circ n}(z),w_2)
\right|
\le
\rho_{\mathbb{D}}(w_1,w_2).
\end{equation*}
Thus, an $O(1)$ bound for one base point implies the same bound for every base point.

The property is also independent of the initial point $z$. By the
Schwarz--Pick contraction property,
\begin{equation*}
\rho_{\mathbb{D}}(g^{\circ n}(z),g^{\circ n}(z'))
\le
\rho_{\mathbb{D}}(z,z')
\quad\text{for all }n\ge0.
\end{equation*}
Consequently,
\begin{equation*}
\left|
\rho_{\mathbb{D}}(g^{\circ n}(z),w)
-
\rho_{\mathbb{D}}(g^{\circ n}(z'),w)
\right|
\le
\rho_{\mathbb{D}}(g^{\circ n}(z),g^{\circ n}(z'))
\le
\rho_{\mathbb{D}}(z,z'),
\end{equation*}
and hence, an $O(1)$ bound for one initial point implies the same
bound for every $z'\in\mathbb{D}$.
\end{definition}

\begin{definition}
\label{def:non-extremal_rate}
Let $g:\mathbb{D}\to\mathbb{D}$ be hyperbolic with Denjoy--Wolff point
$\tau\in\partial\mathbb{D}$ and boundary multiplier
$\alpha\in(0,1)$. We say that $g$ has a
\emph{non--extremal hyperbolic rate} if it does not have an extremal
hyperbolic rate. Equivalently, for some (equivalently, every)
$z,w\in\mathbb{D}$,
\begin{equation*}
\label{eq:non_extremal_rate}
\limsup_{n\to\infty}
\left(
\rho_{\mathbb{D}}(g^{\circ n}(z),w)
-
n\log\frac{1}{\alpha}
\right)
=
+\infty.
\end{equation*}
Indeed, by Proposition~\ref{prop:julia-lower-bound}, for every fixed
$z,w\in\mathbb{D}$ the quantity
\begin{equation*}
\rho_{\mathbb{D}}(g^{\circ n}(z),w)
-
n\log\frac{1}{\alpha}
\end{equation*}
is bounded below uniformly in $n$. Consequently, failure of the
boundedness condition in \eqref{eq:extremal_rate} is equivalent to
unboundedness above, which is precisely \eqref{eq:non_extremal_rate}.
The independence of the choice of $z$ and $w$ follows from the
triangle inequality and the Schwarz--Pick contraction argument used
in Definition~\ref{def:extremal_rate}.

Thus, non--extremality does not correspond to a larger asymptotic
linear rate. Rather, it means that the deviation from the linear
model
\begin{equation*}
n\log\frac{1}{\alpha}
\end{equation*}
is unbounded above: along some subsequence $\{n_j\}$,
\begin{equation*}
\rho_{\mathbb{D}}(g^{\circ n_j}(z),w)
-
n_j\log\frac{1}{\alpha}
\longrightarrow+\infty.
\end{equation*}
In view of Proposition~\ref{prop:julia-lower-bound}, the deviation
cannot diverge to $-\infty$.
\end{definition}

\begin{remark}
\label{rem:cruz_comparison}
Here, we compare the Definition \eqref{def:extremal_rate} with the half-plane formulation.
Our definition of extremal hyperbolic rate is consistent, under the
corresponding change of model and normalization, with the notion of
extremality for self-maps of the upper half-plane considered by
Cruz-Zamorano--Zarvalis~\cite[Definition~3.3]{cruz}. There, for a
self-map $f$ of $\mathbb{H}$ with Denjoy--Wolff point at $\infty$,
extremality is characterized by the existence of $\lambda>1$ and
$L\in\mathbb{H}$ such that
\begin{equation*}
\frac{f^{\circ n}(z)}{\lambda^n}\longrightarrow L.
\end{equation*}
Under Koenigs-coordinate normalization, the parameter $\lambda$
corresponds to $1/\alpha$. The equivalence between this Euclidean normalization and the present hyperbolic formulation is established
in Theorem~\ref{thm:koenigs_criterion} below.

The metric formulation used here has the advantage that it is
intrinsic and does not depend on a particular Euclidean model of the
domain. This makes it suitable for subsequent extensions to more
general domains and, in particular, to settings in which a Euclidean
normalization of the orbit is not naturally available.
\end{remark}

In the following remark, we state the relation of our definition with the notion of extremal rate in the continuous dynamics investigated in \cite{cruz26}.
\begin{remark}
\label{rem:continuous_extremal_rate}
The terminology of extremal rate is closely related to the notion introduced by Cruz-Zamorano and Zarvalis
\cite[Section~3.1]{cruz26} for continuous holomorphic dynamics.
More precisely, let $(\phi_t)_{t\geq0}$ be a hyperbolic semigroup in
$\mathbb{H}$ with a Denjoy--Wolff point at $\infty$ and a spectral value
$\lambda>0$. In this setting, it is known that
\begin{equation*}
\liminf_{t\to\infty}
\frac{|\phi_t(z)|}{e^{\lambda t}}>0,
\quad z\in\mathbb{H},
\end{equation*}
so that $e^{\lambda t}$ gives the sharp exponential scale of the escape of the orbits. The semigroup is said to be of extremal rate
when
\begin{equation*}
\lim_{t\to\infty}
\frac{\phi_t(z)}{e^{\lambda t}}
\in\mathbb{H}
\end{equation*}
for some (and hence, all) $z\in\mathbb{H}$.
Thus, in their Euclidean formulation, extremality means that the
orbit attains the slowest possible exponential escape rate without
an additional unbounded multiplicative factor.

Our definition is the intrinsic hyperbolic-metric counterpart of
this notion in discrete dynamics. Indeed, after conjugating a
hyperbolic self-map of $\mathbb{D}$ to the corresponding model in
$\mathbb{H}$, the boundary multiplier $\alpha\in(0,1)$ is related to
the exponential scaling factor by
\begin{equation*}
\lambda=\log\frac{1}{\alpha}.
\end{equation*}
Accordingly, our linear model
\begin{equation*}
n\log\frac{1}{\alpha}
\end{equation*}
is the discrete-time counterpart of the exponential scale
$e^{\lambda t}$. The condition
\begin{equation*}
\rho_{\mathbb{D}}(g^{\circ n}(z),w)
=
n\log\frac{1}{\alpha}+O(1)
\end{equation*}
therefore expresses, in hyperbolic distance, that the orbit has no
unbounded additive excess over the sharp linear escape scale.

In the hyperbolic semigroup setting, the Euclidean formulation
\begin{equation*}
\frac{\phi_t(z)}{e^{\lambda t}}\longrightarrow L\in\mathbb{H}
\end{equation*}
and the corresponding bounded-deviation formulation in hyperbolic distance are equivalent after passing to the appropriate Koenigs
coordinate; see Theorem~\ref{thm:koenigs_criterion} below. Thus, our definition may be viewed as an intrinsic metric reformulation of the
same extremality principle, while being formulated directly for
iterates and without requiring a Euclidean normalization of the
orbit.
\end{remark}

\subsection{Herglotz--Nevanlinna representation}
\label{subsec:herglotz}

We record the two forms of the representation used in
this paper. We first recall the Herglotz representation in $\mathbb{D}$, see
{\cite[Theorem~2.1.1]{bracci2020}}.

\begin{theorem}
\label{thm:herglotz_disk}
Let $p:\mathbb{D}\to\{s\in\mathbb{C}:\operatorname{Re}s\ge 0\}$
be holomorphic.
Then there exists a unique finite non-negative Borel measure
$\mu$ on $\partial\mathbb{D}$ such that
\begin{equation*}
p(z)
=\int_{\partial\mathbb{D}}\frac{\zeta+z}{\zeta-z}\,d\mu(\zeta)
 +i\operatorname{Im}p(0),
\quad z\in\mathbb{D}.
\end{equation*}
\end{theorem}

The next theorem is the Nevanlinna representation in $\mathbb{H}$, see
{\cite{Nev} or \cite{Cau} in the current form}.
\begin{theorem}
\label{thm:nevanlinna}
A function $f:\mathbb{H}\to\mathbb{C}$ is holomorphic with
$\operatorname{Im}f(z)\ge 0$ for all $z\in\mathbb{H}$
if and only if there exist $a\in\mathbb{R}$, $b\ge 0$,
and a positive Borel measure $\nu$ on $\mathbb{R}$
satisfying $\int_{\mathbb{R}}(1+t^2)^{-1}\,d\nu(t)<\infty$,
such that
\begin{equation}
\label{eq:nevanlinna_rep}
f(z)
=a+bz
+\int_{\mathbb{R}}
  \left(\frac{1}{t-z}-\frac{t}{1+t^2}\right)d\nu(t),
\quad z\in\mathbb{H}.
\end{equation}
The triple $(a,b,\nu)$ is unique.
\end{theorem}

\subsection{Koenigs linearization}
\label{subsec:koenigs}

Let $g:\mathbb{D}\to\mathbb{D}$ be a hyperbolic holomorphic self-map with
Denjoy--Wolff point $\tau\in\partial\mathbb{D}$ and
multiplier $\alpha\in(0,1)$.
A (local) Koenigs function for $g$ is a nonzero holomorphic
function $h$ defined in a Stolz region $S(\tau,M)$ at $\tau$ such that
\begin{equation}
\label{eq:koenigs_func}
h(g(z))=\alpha^{-1}h(z),
\quad
|h(z)|\to\infty
\quad\text{as } z\to\tau
\text{ non-tangentially}.
\end{equation}
See, for example, \cite[p.~271, Notes~9.9]{bracci2020}. We use the reciprocal
normalization~\eqref{eq:koenigs_func}, which is more convenient for
Lemma~\ref{lem:koenigs_metric} below.

Iterating~\eqref{eq:koenigs_func}, one obtains by induction

\begin{equation*}
h(g^{\circ n}(z))
=
\alpha^{-n}h(z),
\quad n\ge1.
\end{equation*}

The following lemma relates the boundary behavior of $h$ to the hyperbolic distance near the Denjoy--Wolff point.

\begin{lemma}
\label{lem:koenigs_metric}
Let $g:\mathbb{D}\to\mathbb{D}$ be a hyperbolic holomorphic
self-map with Denjoy--Wolff point $\tau\in\partial\mathbb{D}$ and
Koenigs function $h$. Suppose that there exists a conformal map
$\phi_\tau:\mathbb{D}\to\mathbb{H}$ satisfying
$\phi_\tau(\tau)=\infty$ and
\begin{equation*}
\label{eq:koenigs_angular}
L:=\angle\lim_{z\to\tau}
\frac{h(z)}{\phi_\tau(z)}
\end{equation*}
with $0<|L|<\infty$. Then, for every fixed $w\in\mathbb{D}$,
\begin{equation}
\label{eq:koenigs_rho}
\rho_{\mathbb{D}}(z,w)
=
\log|h(z)|+O(1)
\quad
(z\to\tau\ \text{non--tangentially}).
\end{equation}
More precisely, for every $M>1$, the $O(1)$ term in
\eqref{eq:koenigs_rho} is bounded uniformly for
$z\in S(\tau,M)$ sufficiently close to $\tau$.
\end{lemma}

\begin{proof}
Fix $M>1$. Since
\begin{equation*}
\angle\lim_{z\to\tau}\frac{h(z)}{\phi_\tau(z)}
=L\neq0,
\end{equation*}
we have
\begin{equation*}
|h(z)|\asymp |\phi_\tau(z)|,
\quad
(z\to\tau,\ z\in S(\tau,M)).
\end{equation*}
Hence,
\begin{equation}
\label{eq:h_phi_log}
\log|h(z)|
=
\log|\phi_\tau(z)|+O(1),
\quad
(z\to\tau,\ z\in S(\tau,M)).
\end{equation}
Since $\phi_\tau$ is a conformal map from $\mathbb{D}$ onto
$\mathbb{H}$ with $\phi_\tau(\tau)=\infty$, it has the form
\begin{equation*}
\phi_\tau(z)
=
a\,i\frac{\tau+z}{\tau-z}+b,
\quad a>0,\quad b\in\mathbb{R}.
\end{equation*}
Consequently,
\begin{equation*}
|\phi_\tau(z)|\asymp\frac{1}{|z-\tau|},
\quad (z\to\tau).
\end{equation*}
On the Stolz region $S(\tau,M)$,
\begin{equation*}
1-|z|
\le |z-\tau|
\le M(1-|z|),
\end{equation*}
and therefore
\begin{equation*}
|z-\tau|\asymp1-|z|\asymp1-|z|^2.
\end{equation*}
It follows that
\begin{equation}
\label{eq:phi_boundary_comparison}
\log|\phi_\tau(z)|
=
\log\frac{1}{1-|z|^2}+O(1),
\quad
(z\to\tau,\ z\in S(\tau,M)).
\end{equation}
Combining \eqref{eq:h_phi_log} and
\eqref{eq:phi_boundary_comparison}, we obtain
\begin{equation*}
\log|h(z)|
=
\log\frac{1}{1-|z|^2}+O(1).
\end{equation*}
On the other hand, Lemma~\ref{lem:boundary_asymp} gives, for fixed
$w\in\mathbb{D}$,
\begin{equation*}
\rho_{\mathbb{D}}(z,w)
=
\log\frac{1}{1-|z|^2}+O(1).
\end{equation*}
Subtracting the two estimates yields
\begin{equation*}
\rho_{\mathbb{D}}(z,w)
=
\log|h(z)|+O(1),
\end{equation*}
as claimed.
\end{proof}

%%=============================================================
\section{Extremal Rate in the Unit Disk and Beyond}
\label{sec:uds}
%%=============================================================

We begin by establishing the precise relationship between
the angular derivative of a hyperbolic self-map of
$\mathbb{D}$ and the asymptotic behavior of its conjugate in the upper half-plane under the Cayley transform.
This multiplier inversion is the key computational step that connects the theory of the disk and the half-plane. We then prove the Herglotz criterion
(Theorem~\ref{thm:herglotz_criterion}), characterizing
extremal rate via the logarithmic integrability of the
boundary measure $\sigma$, by reducing the disk problem
to the half-plane criterion of Cruz-Zamorano--Zarvalis
\cite{cruz} through an explicit kernel transform computation.
The equivalent Koenigs characterization
(Theorem~\ref{thm:koenigs_criterion}) is then derived,
identifying extremality with the non-degeneracy of the
linearizing coordinate at Denjoy--Wolff point.
The section closes by extending these results to simply
connected and multiply connected hyperbolic domains via
conformal conjugation and universal covering, respectively.

Throughout, the unifying principle is that extremality is
a boundary regularity condition: it fails precisely when
the boundary mass of $\sigma$ concentrates near $\tau$
at a logarithmically non-integrable rate, or equivalently,
when the Koenigs function loses its linear scaling in the
boundary variable.

\subsection{Multiplier inversion under the Cayley transform}
\label{subsec:cayley}

The following lemma shows that the Cayley transform
\begin{equation}\label{cayley tr}
\phi(z)=i\frac{1+z}{1-z}
\end{equation}
inverts the multiplier at the Denjoy--Wolff point: a hyperbolic self-map
$g$ of $\mathbb{D}$ with angular derivative $\alpha\in(0,1)$ at $\tau=1$
conjugates to a self-map $f$ of $\mathbb{H}$ satisfying
$f(w)/w\to 1/\alpha>1$ at infinity, a fact that is essential for applying
half-plane representation theory, but requires a careful proof via the
Julia--Wolff--Carath\'eodory chain rule.

\begin{lemma}
\label{lem:multiplier_inversion}
Let $g:\mathbb{D}\to\mathbb{D}$ be a hyperbolic holomorphic self-map with
the Denjoy--Wolff point $\tau=1$ and the angular derivative
$\alpha\in(0,1)$. Let $\phi$ be the Cayley transform defined in \eqref{cayley tr}, and define the conjugate map
$f=\phi\circ g\circ\phi^{-1}:\mathbb{H}\to\mathbb{H}$.
Then
\begin{enumerate}
\item[\emph{(i)}] $f$ has its Denjoy--Wolff point at $\infty$.
\item[\emph{(ii)}]
\begin{equation*}
\lim_{w\to\infty}\frac{f(w)}{w}
=
\lambda:=\frac1\alpha>1
\end{equation*}
as $w\to\infty$ non-tangentially in $\mathbb{H}$.
\end{enumerate}
\end{lemma}

\begin{proof}
We first verify that $f$ maps $\mathbb{H}$ to $\mathbb{H}$. For any
$w\in\mathbb{H}$, its preimage under the Cayley transform is
$z:=\phi^{-1}(w)=(w-i)/(w+i)$, which lies in $\mathbb{D}$ since
$|w-i|<|w+i|$ whenever $\operatorname{Im}w>0$. Because $g$ maps
$\mathbb{D}$ into $\mathbb{D}$, it follows that $g(z)\in\mathbb{D}$.
Using the standard algebraic identity
\begin{equation*}
\operatorname{Im}\phi(u)
=
\operatorname{Im}\!\left(\frac{i(1+u)}{1-u}\right)
=
\frac{1-|u|^2}{|1-u|^2},
\end{equation*}
valid for all $u\in\mathbb{D}$, we obtain
\begin{equation*}
\operatorname{Im}f(w)
=
\operatorname{Im}\phi(g(z))
=
\frac{1-|g(z)|^2}{|1-g(z)|^2}
>0.
\end{equation*}
Thus, $f(w)\in\mathbb{H}$, confirming that
$f:\mathbb{H}\to\mathbb{H}$.

By induction, the iterates satisfy
$f^{\circ n}=\phi\circ g^{\circ n}\circ\phi^{-1}$ for all $n\ge1$.
Fix an arbitrary point $w_0=\phi(z_0)\in\mathbb{H}$.
Because $1$ is the Denjoy--Wolff point of $g$, we have
$g^{\circ n}(z_0)\to1$ as $n\to\infty$. Consequently,
$1+g^{\circ n}(z_0)\to2$ while
$1-g^{\circ n}(z_0)\to0$, and therefore
\begin{equation*}
|f^{\circ n}(w_0)|
=
|\phi(g^{\circ n}(z_0))|
=
\frac{|1+g^{\circ n}(z_0)|}
{|1-g^{\circ n}(z_0)|}
\longrightarrow+\infty.
\end{equation*}
Hence,
$f^{\circ n}(w_0)\to\infty$,
establishing~\emph{(i)}.

For statement~\emph{(ii)}, the classical
Julia--Wolff--Carath\'eodory theorem applied to $g$ at $\tau=1$
asserts that
\begin{equation}
\label{eq:jwc_proof}
\lim_{z\to1}
\frac{1-g(z)}{1-z}
=
\alpha.
\end{equation}
Setting $w=\phi(z)$, we invert the relation to write
$z=(w-i)/(w+i)$, which implies
\begin{equation*}
1-z=\frac{2i}{w+i}.
\end{equation*}
Similarly, $g(z)=\phi^{-1}(f(w))$
gives
\begin{equation*}
1-g(z)=\frac{2i}{f(w)+i}.
\end{equation*}
Substituting these expressions into
\eqref{eq:jwc_proof} yields
\begin{equation*}
\alpha
=
\lim_{w\to\infty}
\frac{1-g(\phi^{-1}(w))}
{1-\phi^{-1}(w)}
=
\lim_{w\to\infty}
\frac{2i/(f(w)+i)}
{2i/(w+i)}
=
\lim_{w\to\infty}
\frac{w+i}{f(w)+i}.
\end{equation*}
Since $f(w)\to\infty$ as $w\to\infty$, we have
$f(w)+i\sim f(w)$ and
$w+i\sim w$. Hence,
\begin{equation*}
\alpha
=
\lim_{w\to\infty}
\frac{w}{f(w)},
\end{equation*}
or equivalently,
\begin{equation*}
\frac{f(w)}{w}\longrightarrow\frac1\alpha=:\lambda.
\end{equation*}

It remains to verify that the Cayley transform preserves
non-tangential approach regions.
Let $w=\phi(z)$ for $z\in S(1,M)$.
Recall from the geometric relations in
Lemma~\ref{lem:koenigs_metric} that
$|1-z|\asymp1-|z|^2$
inside $S(1,M)$.
Since
$|w|\asymp1/|1-z|$
and
\begin{equation*}
\operatorname{Im}w
=
\frac{1-|z|^2}{|1-z|^2},
\end{equation*}
the ratio satisfies
\begin{equation*}
\frac{|w|}{\operatorname{Im}w}
\asymp
\frac{1/|1-z|}
{(1-|z|^2)/|1-z|^2}
=
\frac{|1-z|}{1-|z|^2}
\le C(M)<\infty.
\end{equation*}
For $c>0$, let
\begin{equation*}
\Gamma(c)
:=
\{\,w\in\mathbb H:
|\operatorname{Re}w|
<
c\,\operatorname{Im}w
\,\},
\end{equation*}
the standard non-tangential sector at $\infty$ in $\mathbb H$.
Because $|\operatorname{Re}w|\le|w|$, there exists a constant
$c_1:=c_1(M)>0$ such that
$|\operatorname{Re}w|
\le
c_1\operatorname{Im}w$,
that is,
$\phi(S(1,M))
\subseteq
\Gamma(c_1)$.

Conversely, fix $c_2>0$ and let
$w\in\Gamma(c_2)$, and set
$z=\phi^{-1}(w)$.
From
$1-z=2i/(w+i)$,
\begin{equation*}
|1-z|
=
\frac2{|w+i|},
\end{equation*}
while
\begin{equation*}
\operatorname{Im}w
=
\frac{1-|z|^2}{|1-z|^2}.
\end{equation*}
Since
$|w|
\le
\sqrt{1+c_2^2}\,\operatorname{Im}w$,
we also have
$|w+i|
\asymp
\operatorname{Im}w$,
where the implicit constants depend only on $c_2$.
Therefore
\begin{equation*}
|1-z|
\asymp
\frac1{\operatorname{Im}w},
\end{equation*}
and
\begin{equation*}
1-|z|^2
=
\operatorname{Im}w\,|1-z|^2
\asymp
\frac1{\operatorname{Im}w}.
\end{equation*}
Hence,
\begin{equation*}
|1-z|
\asymp
1-|z|^2,
\end{equation*}
which is precisely the Stolz condition.
Thus,
$\phi^{-1}(\Gamma(c_2))$
is eventually contained in a Stolz region
$S(1,M)$ for some $M$ depending only on $c_2$.

Combining both inclusions, approaching $\infty$
non-tangentially in $\mathbb{H}$ is equivalent to
approaching $1$ non-tangentially in $\mathbb{D}$.
Therefore, the limit obtained from the
Julia--Wolff--Carath\'eodory theorem is valid for every
non-tangential approach to $\infty$ in $\mathbb{H}$,
which completes the proof of~\emph{(ii)}.
\end{proof}

\subsection{The Herglotz criterion}
\label{subsec:herglotz_criterion}
With the multiplier inversion of Lemma~\ref{lem:multiplier_inversion} in hand, we are now in a position to prove one of the main results of this section. The theorem consists of two core parts.

The first part establishes that every hyperbolic self-map $g$ of $\mathbb{D}$ admits a canonical integral representation of the form~\eqref{eq:disk_rep} below, in which the boundary dynamics are explicitly encoded by a positive Borel measure $\sigma$ on $\partial\mathbb{D}\setminus\{\tau\}$ along with the angular multiplier $\alpha$. This representation serves as the disk analogue of the classical Nevanlinna representation~\eqref{eq:nevanlinna_rep} for holomorphic self-maps of the upper half-plane, and is derived from it via a direct Cayley-transform change of variables.

The second part identifies the precise constraint on the boundary mass distribution $\sigma$ that governs extremality: the $O(1)$ error term in the hyperbolic distance asymptotic is uniformly bounded if and only if $\sigma$ satisfies the logarithmic integrability condition~\eqref{eq:extremal_condition} near $\tau$. This condition measures the exact concentration of boundary mass near the Denjoy--Wolff point and is strictly stronger than requiring the total mass of $\sigma$ to be finite, as highlighted by the examples in the following subsection.

%-------------------------------------
\begin{theorem}
\label{thm:herglotz_criterion}
Let $g:\mathbb{D}\to\mathbb{D}$ be a hyperbolic holomorphic self-map
with Denjoy--Wolff point $\tau\in\partial\mathbb{D}$ and angular
multiplier $\alpha:=g'(\tau)$.

Then there exist a real constant $b\in\mathbb{R}$ and a finite positive
Borel measure $\sigma$ on
$\partial\mathbb{D}\setminus\{\tau\}$ such that
\begin{equation*}
\label{eq:sigma_integrability}
\sigma\bigl(\partial\mathbb{D}\setminus\{\tau\}\bigr)<\infty
\end{equation*}
and
\begin{equation}
\label{eq:disk_rep}
\frac{1+\overline{\tau}g(z)}
     {1-\overline{\tau}g(z)}
=
\frac{1}{\alpha}
\frac{1+\overline{\tau}z}
     {1-\overline{\tau}z}
+ib
+\int_{\partial\mathbb{D}\setminus\{\tau\}}
\frac{\zeta+z}{\zeta-z}\,d\sigma(\zeta),
\quad z\in\mathbb{D}.
\end{equation}
Moreover, $g$ has an extremal hyperbolic rate if and only if
\begin{equation}
\label{eq:extremal_condition}
\int_{\partial\mathbb{D}\setminus\{\tau\}}
\log\frac{1}{|\zeta-\tau|}\,d\sigma(\zeta)
<\infty.
\end{equation}
\end{theorem}

\begin{proof}
We first consider the case $\tau=1$. The general case follows by
conjugation with a rotation of $\mathbb{D}$.

Let $\phi:\mathbb{D}\to\mathbb{H}$ be the Cayley transform \eqref{cayley tr}
and define
\begin{equation*}
f=\phi\circ g\circ\phi^{-1}:\mathbb{H}\to\mathbb{H}.
\end{equation*}
By Lemma~\ref{lem:multiplier_inversion}, $f$ is hyperbolic with
Denjoy--Wolff point $\infty$ and angular derivative
\begin{equation*}
f'(\infty)=\frac{1}{\alpha}>1.
\end{equation*}

By the Nevanlinna representation, there exist $a\in\mathbb{R}$ and a
positive Borel measure $\nu$ on $\mathbb{R}$ satisfying
\begin{equation*}
\int_{\mathbb{R}}\frac{d\nu(t)}{1+t^2}<\infty
\end{equation*}
such that
\begin{equation}
\label{eq:f_rep}
f(w)
=
a+\frac{1}{\alpha}w
+\int_{\mathbb{R}}
\left(
\frac{1}{t-w}-\frac{t}{1+t^2}
\right)d\nu(t),
\quad w\in\mathbb{H}.
\end{equation}

For $\zeta\in\partial\mathbb{D}\setminus\{1\}$, put
\begin{equation*}
t=\phi(\zeta)
=i\frac{1+\zeta}{1-\zeta}\in\mathbb{R},
\end{equation*}
and, for $z\in\mathbb{D}$, put
\begin{equation*}
w=\phi(z)=i\frac{1+z}{1-z}.
\end{equation*}
A direct calculation gives
\begin{equation}
\label{eq:kernel_id}
\frac{1}{t-w}-\frac{t}{1+t^2}
=
\frac{i}{1+t^2}\frac{\zeta+z}{\zeta-z}.
\end{equation}
Indeed,
\begin{equation*}
\frac{1}{t-w}-\frac{t}{1+t^2}
=
\frac{1+tw}{(t-w)(1+t^2)},
\end{equation*}
while
\begin{equation*}
\frac{1+tw}{t-w}
=
i\frac{\zeta+z}{\zeta-z}.
\end{equation*}

Let $\sigma$ be the measure on
$\partial\mathbb{D}\setminus\{1\}$ obtained by pushing forward the
measure $(1+t^2)^{-1}d\nu(t)$ under the inverse Cayley transform
$t\mapsto\zeta=\phi^{-1}(t)$. Equivalently,
\begin{equation}
\label{eq:sigma_definition}
d\sigma(\zeta)
=
\frac{1}{1+t^2}\,d\nu(t),
\quad
t=\phi(\zeta).
\end{equation}
Since
\begin{equation*}
\frac{1}{1+t^2}
=
\frac{|1-\zeta|^2}{4},
\end{equation*}
the measure $\sigma$ is positive. Moreover,
\begin{equation*}
\sigma\bigl(\partial\mathbb{D}\setminus\{1\}\bigr)
=
\int_{\mathbb{R}}\frac{d\nu(t)}{1+t^2}
<\infty.
\end{equation*}
Using \eqref{eq:kernel_id} and
$d\nu(t)=(1+t^2)d\sigma(\zeta)$ in \eqref{eq:f_rep}, we obtain
\begin{equation*}
f(w)
=
a+\frac{1}{\alpha}w
+i\int_{\partial\mathbb{D}\setminus\{1\}}
\frac{\zeta+z}{\zeta-z}\,d\sigma(\zeta).
\end{equation*}
Since
\begin{equation*}
f(w)=i\frac{1+g(z)}{1-g(z)}
\quad\text{and}\quad
w=i\frac{1+z}{1-z},
\end{equation*}
division by $i$ gives
\begin{equation*}
\label{eq:g_rep_tau1}
\frac{1+g(z)}{1-g(z)}
=
\frac{1}{\alpha}\frac{1+z}{1-z}
-ia
+\int_{\partial\mathbb{D}\setminus\{1\}}
\frac{\zeta+z}{\zeta-z}\,d\sigma(\zeta).
\end{equation*}
Thus, \eqref{eq:disk_rep} holds for $\tau=1$, with $b=-a$.

It remains to characterize extremality. In the half-plane setting,
the Herglotz representation
\begin{equation*}
f(w)
=
\frac{1}{\alpha}w+a
+
\int_{\mathbb{R}}
\frac{1+tw}{t-w}\,d\mu(t),
\end{equation*}
where
\begin{equation*}
d\mu(t)=\frac{1}{1+t^2}\,d\nu(t),
\end{equation*}
is related to extremal rate by the criterion
\begin{equation*}
\int_{\mathbb{R}}\log(1+|t|)\,d\mu(t)<\infty.
\end{equation*}
This is precisely the Herglotz characterization of extremal rate for
hyperbolic self-maps of $\mathbb{H}$; see
\cite[Theorem~4.4]{cruz}.
Consequently, in terms of $\nu$, extremality is equivalent to
\begin{equation}
\label{eq:half_plane_cond}
\int_{\mathbb{R}}
\frac{\log(1+|t|)}{1+t^2}\,d\nu(t)
<\infty.
\end{equation}

Since the Cayley transform is a hyperbolic isometry, extremality is
preserved under the conjugation between $\mathbb{D}$ and $\mathbb{H}$.
By \eqref{eq:sigma_definition}, condition \eqref{eq:half_plane_cond}
is equivalent to
\begin{equation*}
\int_{\partial\mathbb{D}\setminus\{1\}}
\log(1+|\phi(\zeta)|)\,d\sigma(\zeta)<\infty.
\end{equation*}
Now
\begin{equation*}
|\phi(\zeta)|
=
\frac{|1+\zeta|}{|1-\zeta|}.
\end{equation*}
As $\zeta\to1$ along $\partial\mathbb{D}$,
$|1+\zeta|\to2$, and hence,
\begin{equation*}
\log(1+|\phi(\zeta)|)
=
\log\frac{1}{|1-\zeta|}+O(1).
\end{equation*}
More precisely, the difference between the two functions is bounded
in a sufficiently small punctured neighborhood of $1$. Since $\sigma$
is finite, their difference is $\sigma$-integrable there. On the
complement of such a neighborhood, both functions are bounded.
Therefore,
\begin{equation*}
\int_{\partial\mathbb{D}\setminus\{1\}}
\log(1+|\phi(\zeta)|)\,d\sigma(\zeta)<\infty
\end{equation*}
if and only if
\begin{equation*}
\int_{\partial\mathbb{D}\setminus\{1\}}
\log\frac{1}{|1-\zeta|}\,d\sigma(\zeta)<\infty.
\end{equation*}
This proves \eqref{eq:extremal_condition} when $\tau=1$.

Finally, for a general $\tau\in\partial\mathbb{D}$, conjugation by the
rotation $R_\tau(z)=\overline{\tau}z$ maps $\tau$ to $1$ and preserves
the hyperbolic metric. Applying the preceding argument to
$R_\tau\circ g\circ R_\tau^{-1}$ and pulling the resulting measure back
under $R_\tau^{-1}$ gives \eqref{eq:disk_rep} and
\eqref{eq:extremal_condition}. The proof is now complete.
\end{proof}

%'''''''''''''''''''''''''''''''''''''''''''''''''''''''''''''''''''''''''''''''

%%%%%%%%%%%%%%%%%%%%%%%%%%%%%%%%%%%%%%%%

%%%%%%%%%%%%%%%%%%%%%%%%%%%%%%%%
%%%%%%%%%%%%%%%%%%%%%%%%%%%%%%%%
\begin{remark}
\label{rem:self_map_constraint}
The representation~\eqref{eq:disk_rep} is not a definition of $g$ but a
structural consequence of the self-map property. More precisely, for a
given hyperbolic self-map $g$, let
\[
f=\phi\circ g\circ\phi^{-1}:\mathbb H\to\mathbb H
\]
be its conjugate in the upper half-plane. The Nevanlinna representation
of $f$ then yields the real constant $b$ and the finite positive Borel
measure $\sigma$ appearing in~\eqref{eq:disk_rep}. Thus, the fact that
$g$ is a self-map of $\mathbb D$ is already encoded in the origin of the
representation and is not an additional condition imposed independently
on $\sigma$.

Consequently, one should not regard an arbitrary pair
$(\alpha,\sigma)$, with $\alpha\in(0,1)$ and $\sigma$ a finite positive
Borel measure on
$\partial\mathbb D\setminus\{\tau\}$, as automatically arising from a
hyperbolic self-map through~\eqref{eq:disk_rep}. The representation
describes the Herglotz--Nevanlinna data associated with a self-map; it is
not, in the present formulation, being used as a converse
parameterization of all such maps.

Finally, the logarithmic integrability condition
\eqref{eq:extremal_condition} is strictly stronger than the finiteness
of the total mass
$\sigma(\partial\mathbb D\setminus\{\tau\})$. Thus, although the
Herglotz measure $\sigma$ is finite for every hyperbolic self-map, its
mass may nevertheless be concentrated sufficiently strongly near the
Denjoy--Wolff point $\tau$ so that
\[
\int_{\partial\mathbb D\setminus\{\tau\}}
\log\frac{1}{|\zeta-\tau|}
\,d\sigma(\zeta)
=\infty.
\]
Such logarithmically non-integrable concentration is precisely what
leads to the non-extremal regime, as illustrated in
Example~\ref{ex:concentrated} below.
\end{remark}

%\subsection{Examples}
The following example shows that when the boundary measure
is spread uniformly around $\partial\mathbb{D}$, with no
concentration near Denjoy--Wolff point, the logarithmic
integrability condition~\eqref{eq:extremal_condition} is
satisfied and the corresponding self-map achieves extremal
hyperbolic rate.
\begin{example}
\label{ex:uniform}
Let $\tau=1$ and let $\sigma$ be the normalized Lebesgue measure
on $\partial\mathbb{D}\setminus\{1\}$, that is,
$d\sigma(\zeta)=d\theta/(2\pi)$, where $\zeta=e^{i\theta}$ and
$\theta\in(-\pi,\pi)$.
It is not difficult to see that this is a finite positive Borel measure with total mass
\begin{equation*}
\sigma(\partial\mathbb{D}\setminus\{1\})
=\frac{1}{2\pi}\int_{-\pi}^{\pi}d\theta
=\frac{1}{2\pi}\cdot 2\pi=1,
\end{equation*}
where the single point $\theta=0$ (corresponding to
$\zeta=1$) has the Lebesgue measure zero and is excluded
without affecting the integral.

We verify condition~\eqref{eq:extremal_condition}.
By the even symmetry of the integrand,
\begin{equation*}
\int_{\partial\mathbb{D}\setminus\{1\}}
\log\frac{1}{|\zeta-1|}\,d\sigma(\zeta)
=\frac{1}{\pi}\int_0^{\pi}
\log\frac{1}{|e^{i\theta}-1|}\,d\theta.
\end{equation*}
The exact identity $|e^{i\theta}-1|=2\sin(\theta/2)$
for $\theta\in(0,\pi]$ gives
$\log(1/|e^{i\theta}-1|)=\log(1/(2\sin(\theta/2)))$.
It is well known that
$\int_0^\pi\log(2\sin(\theta/2))\,d\theta=0$. Therefore,
\begin{equation*}
\int_{\partial\mathbb{D}\setminus\{1\}}
\log\frac{1}{|\zeta-1|}\,d\sigma(\zeta)
=\frac{1}{\pi}\int_0^\pi\log\frac{1}{2\sin(\theta/2)}\,d\theta
=0<\infty.
\end{equation*}
Condition~\eqref{eq:extremal_condition} holds (with the integral equal to zero), therefore, by
Theorem~\ref{thm:herglotz_criterion}, the corresponding hyperbolic self-map $g$
with $\tau=1$ and $\sigma$ equal to the normalized Lebesgue measure has an extremal hyperbolic rate.
\end{example}

The next example demonstrates that concentrating the boundary measure near the Denjoy--Wolff point $\tau$, even while keeping the total mass finite, causes the logarithmic moment to diverge and forces strictly non-extremal behavior.
%%%%%%%%%%%%%%%%%%%%%%%%%%%%%%%%%%%%
\begin{example}
\label{ex:concentrated}
Let $\varepsilon\in(0,1)$ and define a positive Borel measure $\sigma$ on $\partial\mathbb{D}\setminus\{1\}$ via the boundary density
\begin{equation}
\label{eq:conc_measure}
d\sigma(\zeta)
=\frac{1}{2\pi}
\cdot
\frac{\chi_{(-\varepsilon,\varepsilon)\setminus\{0\}}(\theta)}
{|\theta|\log^2(1/|\theta|)}
\,d\theta,
\quad\zeta=e^{i\theta},
\quad\theta\in(-\pi,\pi)\setminus\{0\}.
\end{equation}
For $0 < |\theta| < \varepsilon < 1$, we have $|\theta|\log^2(1/|\theta|)>0$, meaning the density is strictly positive on its support. Thus, the measure $\sigma$ is well-defined and supported on the punctured symmetric arc $\{e^{i\theta}:0 < |\theta|<\varepsilon\} \subset\partial\mathbb{D}\setminus\{1\}$.

Exploiting the even symmetry of the integrand and applying the substitution $u=\log(1/\theta)$, we compute the total mass of the measure:
\begin{equation*}
\sigma(\partial\mathbb{D}\setminus\{1\})
=\frac{1}{\pi}\int_0^{\varepsilon}
\frac{d\theta}{\theta\log^2(1/\theta)}
=\frac{1}{\pi}\int_{\log(1/\varepsilon)}^{\infty}
\frac{du}{u^2}
=\frac{1}{\pi\log(1/\varepsilon)}<\infty,
\end{equation*}
which establishes that $\sigma$ is finite. Next, we check the logarithmic moment. Using the standard boundary distance estimate $|e^{i\theta}-1| = 2\sin(\theta/2)\le\theta$ for $\theta\in(0,\varepsilon)$, it follows that $\log(1/|e^{i\theta}-1|) \ge \log(1/\theta)$. Therefore:
\begin{equation*}
\int_{\partial\mathbb{D}\setminus\{1\}}
\log\frac{1}{|\zeta-1|}\,d\sigma(\zeta)
\ge\frac{1}{\pi}\int_0^{\varepsilon}
\frac{\log(1/\theta)}{\theta\log^2(1/\theta)}\,d\theta
=\frac{1}{\pi}\int_0^{\varepsilon}
\frac{d\theta}{\theta\log(1/\theta)}
=\frac{1}{\pi}\int_{\log(1/\varepsilon)}^{\infty}
\frac{du}{u}=+\infty.
\end{equation*}
Because the structural condition~\eqref{eq:extremal_condition} fails, Theorem~\ref{thm:herglotz_criterion} implies that the corresponding hyperbolic self-map $g$ is strictly non-extremal. While its iterates still converge to $\tau=1$ by the classical Denjoy--Wolff theorem, the boundary measure concentrates too much mass near the singularity at $\theta=0$ to sustain an extremal hyperbolic rate.
\end{example}

\begin{remark}
The density chosen in~\eqref{eq:conc_measure} sits precisely at the critical threshold separating extremal and non-extremal behavior. For the parametric family of boundary densities given by $\rho_p(\theta) = (|\theta|\log^p(1/|\theta|))^{-1}$ with $p>1$, the total mass of the corresponding measure remains finite for all values of $p$. However, the logarithmic moment is finite if and only if $p>2$. Consequently, for $p>2$, the generated map $g$ exhibits an extremal hyperbolic rate, while for $1<p\le 2$, the map is non-extremal. Example~\ref{ex:concentrated} maps directly to the critical limit case where $p=2$.
\end{remark}

%%%%%%%%%%%%%%%%%%%%%%%%%%%%%%%%%%%%%%%%%%%%%%%%%%%%%%%%%%%%%%%%5

\subsection{The Koenigs characterization}
\label{subsec:koenigs_characterization}

The following result gives an equivalent characterization of the
extremal hyperbolic rate in terms of the boundary regularity of the
Koenigs function. In particular, extremality is equivalent to the
conformality of the Koenigs function at the Denjoy--Wolff point.

\begin{theorem}
\label{thm:koenigs_criterion}
Let $g:\mathbb{D}\to\mathbb{D}$ be a hyperbolic holomorphic self-map
with Denjoy--Wolff point $\tau\in\partial\mathbb{D}$ and angular
multiplier $\alpha\in(0,1)$. Let
$h:\mathbb{D}\to\mathbb{H}$ be a Koenigs function of $g$, normalized by
\begin{equation*}
h\circ g=\alpha^{-1}h.
\end{equation*}
Then the following are equivalent:
\begin{enumerate}
\item[\emph{(i)}] $g$ has an extremal hyperbolic rate;
\item[\emph{(ii)}] $h$ is conformal at $\tau$, in the sense that
for some (equivalently, every) conformal map
$\phi_\tau:\mathbb{D}\to\mathbb{H}$ satisfying
$\phi_\tau(\tau)=\infty$,
\begin{equation*}
\label{eq:koenigs_boundary_ratio}
\angle\lim_{z\to\tau}
\frac{h(z)}{\phi_\tau(z)}
\in\mathbb{C}\setminus\{0\}.
\end{equation*}
\end{enumerate}
\end{theorem}

\begin{proof}
We first consider the case $\tau=1$. Let $\phi$ be the Cayley map defined in \eqref{cayley tr}
and define
\begin{equation*}
f=\phi\circ g\circ\phi^{-1}:\mathbb{H}\to\mathbb{H}.
\end{equation*}
By Lemma~\ref{lem:multiplier_inversion}, $f$ is hyperbolic with
Denjoy--Wolff point $\infty$ and multiplier
\begin{equation*}
f'(\infty)=\frac1\alpha>1.
\end{equation*}
Set
$H=h\circ\phi^{-1}.$
Then
$H\circ f=\alpha^{-1}H$,
so $H$ is a Koenigs function for $f$.

\medskip
\noindent\emph{(ii) $\Rightarrow$ (i).}
Assume that
\begin{equation*}
L:=\angle\lim_{z\to1}\frac{h(z)}{\phi(z)}
\end{equation*}
exists and satisfies $0<|L|<\infty$. Let
$z_n=g^{\circ n}(z)$ for a fixed $z\in\mathbb{D}$. Since $g$ is
hyperbolic, $z_n\to1$ non--tangentially. Hence, by
Lemma~\ref{lem:koenigs_metric},
\begin{equation*}
\rho_{\mathbb{D}}(z_n,w)
=
\log|h(z_n)|+O(1).
\end{equation*}
The Koenigs functional equation gives
\begin{equation*}
h(z_n)=\alpha^{-n}h(z),
\end{equation*}
and therefore
\begin{align*}
\rho_{\mathbb{D}}(z_n,w)
&=
\log|\alpha^{-n}h(z)|+O(1)
=
n\log\frac1\alpha+\log|h(z)|+O(1)\\
&=
n\log\frac1\alpha+O(1).
\end{align*}
Thus, $g$ has an extremal hyperbolic rate.

\medskip
\noindent\emph{(i) $\Rightarrow$ (ii).}
Conversely, suppose that $g$ has an extremal hyperbolic rate. Since
$\phi$ is a hyperbolic isometry, the conjugate map $f$ also has an
extremal hyperbolic rate. The Koenigs function of $f$ is $H$ and
satisfies
\begin{equation*}
H\circ f=\alpha^{-1}H.
\end{equation*}
By the Koenigs characterization for hyperbolic self-maps of
$\mathbb{H}$, see~\cite[Theorem~5.2]{cruz}, extremality of $f$ is
equivalent to conformality of $H$ at infinity, namely
\begin{equation*}
\angle\lim_{w\to\infty}\frac{H(w)}{w}
=L
\in\mathbb{C}\setminus\{0\}.
\end{equation*}
Since $H(\phi(z))=h(z)$, this is equivalent to
\begin{equation*}
\angle\lim_{z\to1}\frac{h(z)}{\phi(z)}
=
L
\in\mathbb{C}\setminus\{0\}.
\end{equation*}

For a general $\tau\in\partial\mathbb{D}$, conjugation by the rotation
$R_\tau(z)=\overline{\tau}z$ reduces the assertion to the case
$\tau=1$.

Finally, the condition is independent of the chosen conformal map
$\phi_\tau$. Indeed, if $\psi_\tau:\mathbb{D}\to\mathbb{H}$ is another
such map, then
\begin{equation*}
\psi_\tau\circ\phi_\tau^{-1}(w)=aw+b,
\quad a>0,\quad b\in\mathbb{R}.
\end{equation*}
Thus,
\begin{equation*}
\frac{h(z)}{\psi_\tau(z)}
=
\frac{h(z)/\phi_\tau(z)}
     {a+b/\phi_\tau(z)},
\end{equation*}
and since $\phi_\tau(z)\to\infty$ non--tangentially as
$z\to\tau$,
\begin{equation*}
\angle\lim_{z\to\tau}
\frac{h(z)}{\psi_\tau(z)}
=
\frac{L}{a}.
\end{equation*}
Hence, the existence of a finite nonzero angular limit is independent
of the choice of $\phi_\tau$.
\end{proof}

We continue with the following example illustrating the Koenigs
characterization for a classical linear fractional model.

%%%%%%%%%%%%%%%%%%%%%%%%%%%%%%%%%%%%
\begin{example}
\label{ex:lft_koenigs}

Fix $\alpha\in(0,1)$ and consider the linear fractional self-map \eqref{eq:sharpness_automorphism}
\begin{equation*}
g(z)
=
\frac{(1+\alpha)z+(1-\alpha)}
     {(1-\alpha)z+(1+\alpha)},
\quad z\in\mathbb{D}.
\end{equation*}
We recal that
\begin{equation*}
g(1)=1,
\quad
g'(1)=\alpha,
\end{equation*}
and $g$ is a hyperbolic self-map of $\mathbb{D}$ with Denjoy--Wolff
point $\tau=1$ and angular multiplier $\alpha$.

Let $\phi(z)$
be the standard Cayley transform from $\mathbb{D}$ onto $\mathbb{H}$ defined in \eqref{cayley tr}.
A direct calculation shows that
\begin{equation*}
\phi\circ g=\frac1\alpha\,\phi,
\end{equation*}
and hence, the conjugated map
\begin{equation*}
f=\phi\circ g\circ\phi^{-1}:\mathbb{H}\to\mathbb{H}
\end{equation*}
is the dilation
\begin{equation*}
f(w)=\frac{w}{\alpha}.
\end{equation*}
In particular,
\begin{equation*}
f^{\circ n}(w)=\frac{w}{\alpha^n}.
\end{equation*}
The Koenigs functional equation for $g$ is
\begin{equation*}
h\circ g=\alpha^{-1}h.
\end{equation*}
Consequently, the natural Koenigs function in this model is simply
\begin{equation*}
h(z)=\phi(z)=i\frac{1+z}{1-z}.
\end{equation*}
Indeed,
\begin{equation*}
h(g(z))
=
\phi(g(z))
=
f(\phi(z))
=
\frac{\phi(z)}{\alpha}
=
\alpha^{-1}h(z).
\end{equation*}
We now verify condition~\emph{(ii)} of
Theorem~\ref{thm:koenigs_criterion}. Since $h=\phi$, we may choose
the conformal reference map $\phi_\tau=\phi.$ Then
\begin{equation*}
\angle\lim_{z\to1}
\frac{h(z)}{\phi_\tau(z)}
=
\angle\lim_{z\to1}
\frac{\phi(z)}{\phi(z)}
=
1.
\end{equation*}
Thus, the angular limit exists and satisfies
\begin{equation*}
0<|L|=1<\infty.
\end{equation*}
Therefore, condition~\emph{(ii)} of
Theorem~\ref{thm:koenigs_criterion} holds, and hence, $g$ has an
extremal hyperbolic rate.

In fact, this example makes the metric mechanism particularly
transparent. For any fixed $w\in\mathbb{D}$ and $z_0\in\mathbb{D}$,
the iterates satisfy
\begin{equation*}
g^{\circ n}(z_0)
=
\phi^{-1}\!\left(\frac{\phi(z_0)}{\alpha^n}\right).
\end{equation*}
Since $\phi$ is a hyperbolic isometry,
\begin{equation*}
\rho_{\mathbb{D}}
\bigl(g^{\circ n}(z_0),w\bigr)
=
\rho_{\mathbb{H}}
\left(
\frac{\phi(z_0)}{\alpha^n},
\phi(w)
\right).
\end{equation*}
Moreover,
\begin{equation*}
\frac{h(z_n)}{\phi(z_n)}\equiv1,
\quad
z_n=g^{\circ n}(z_0),
\end{equation*}
so Lemma~\ref{lem:koenigs_metric} gives
\begin{equation*}
\rho_{\mathbb{D}}(z_n,w)
=
\log|h(z_n)|+O(1).
\end{equation*}
Using
\begin{equation*}
h(z_n)=\alpha^{-n}h(z_0),
\end{equation*}
we obtain
\begin{align*}
\rho_{\mathbb{D}}
\bigl(g^{\circ n}(z_0),w\bigr)
&=
\log\left|\alpha^{-n}h(z_0)\right|+O(1)
=
n\log\frac1\alpha+\log|h(z_0)|+O(1)\\
&=
n\log\frac1\alpha+O(1).
\end{align*}
Hence, the orbit realizes precisely the extremal linear growth rate.

Finally, the example also illustrates the independence of the
reference conformal map in Theorem~\ref{thm:koenigs_criterion}. If
$\widetilde{\phi}_\tau:\mathbb{D}\to\mathbb{H}$ is any other conformal
map with $\widetilde{\phi}_\tau(1)=\infty$, then
\begin{equation*}
\widetilde{\phi}_\tau(z)=a\phi(z)+b,
\quad a>0,\quad b\in\mathbb{R},
\end{equation*}
and therefore
\begin{equation*}
\angle\lim_{z\to1}
\frac{h(z)}{\widetilde{\phi}_\tau(z)}
=
\frac1a.
\end{equation*}
Thus, the angular limit remains finite and nonzero, as required.
\end{example}
%%%%%%%%%%%%%%%%%%%%%%%%%%%%%%%%%%%%

\subsection{Extension to simply connected hyperbolic domains}
\label{subsec:simply_connected}

The characterizations established in
Theorems~\ref{thm:herglotz_criterion} and
\ref{thm:koenigs_criterion} extend naturally from the unit disk to
arbitrary simply connected hyperbolic domains. The reason is that a
Riemann mapping is an exact isometry for the intrinsic hyperbolic
metrics. Consequently, both the extremal rate and its characterizations
can be transferred to the disk by conformal conjugation.

For a hyperbolic domain $\Omega$, we denote by
$\lambda_\Omega$ the density of the Poincar\'e metric of curvature $-1$
and by $\rho_\Omega$ the corresponding Poincar\'e distance, namely
\begin{equation*}
\rho_\Omega(z,w)
=
\inf_\gamma\int_\gamma
\lambda_\Omega(\xi)\,|d\xi|.
\end{equation*}

\begin{theorem}
\label{thm:simply_connected}
Let $\Omega\subsetneq\mathbb C$ be a simply connected proper domain,
and let $\rho_\Omega$ denote its Poincar\'e hyperbolic metric. Let $\psi:\Omega\to\mathbb{D}$ be a conformal map, and let
$g:\Omega\to\Omega$ be a holomorphic self-map with no fixed point in
$\Omega$. Define
\begin{equation}\label{g tilde}
\widetilde g
=
\psi\circ g\circ\psi^{-1}:\mathbb{D}\to\mathbb{D}.
\end{equation}

Suppose that $\widetilde g$ is hyperbolic, with Denjoy--Wolff point
$\tau\in\partial\mathbb{D}$ and angular multiplier
$\alpha\in(0,1)$.

Then the following conditions are equivalent:
\begin{enumerate}
\item[\emph{(i)}]
$g$ has an extremal hyperbolic rate in $\Omega$, that is, for some
(equivalently, every) $z,w\in\Omega$,
\begin{equation*}
\label{eq:omega_extremal_rate}
\rho_\Omega\bigl(g^{\circ n}(z),w\bigr)
=
n\log\frac1\alpha+O(1),
\quad (n\to\infty).
\end{equation*}

\item[\emph{(ii)}]
$\widetilde g$ has an extremal hyperbolic rate in $\mathbb{D}$.

\item[\emph{(iii)}]
The Herglotz measure $\sigma$ associated with $\widetilde g$ in
Theorem~\ref{thm:herglotz_criterion} satisfies
\begin{equation*}
\label{eq:omega_extremal_condition}
\int_{\partial\mathbb{D}\setminus\{\tau\}}
\log\frac1{|\zeta-\tau|}
\,d\sigma(\zeta)
<\infty.
\end{equation*}
\end{enumerate}

In particular, the extremal hyperbolic rate is intrinsic to
$\Omega$ and does not depend on the choice of the Riemann mapping
$\psi$.
\end{theorem}

\begin{proof}
Since $\psi:\Omega\to\mathbb{D}$ is a conformal equivalence, it is an
isometry between the intrinsic hyperbolic metrics:
\begin{equation*}
\label{eq:riemann_hyperbolic_isometry}
\rho_\Omega(z,w)
=
\rho_{\mathbb{D}}\bigl(\psi(z),\psi(w)\bigr),
\quad z,w\in\Omega.
\end{equation*}
Moreover,
\begin{equation*}
\psi\circ g^{\circ n}
=
\widetilde g^{\circ n}\circ\psi,
\end{equation*}
where $\widetilde g$ defined in \eqref{g tilde}.
Hence, for every $z,w\in\Omega$,
\begin{equation*}
\label{eq:orbit_conjugacy}
\rho_\Omega\bigl(g^{\circ n}(z),w\bigr)
=
\rho_{\mathbb{D}}
\bigl(\widetilde g^{\circ n}(\psi(z)),\psi(w)\bigr).
\end{equation*}
Therefore,
\begin{equation*}
\rho_\Omega\bigl(g^{\circ n}(z),w\bigr)
=
n\log\frac1\alpha+O(1)
\end{equation*}
if and only if
\begin{equation*}
\rho_{\mathbb{D}}
\bigl(\widetilde g^{\circ n}(\psi(z)),\psi(w)\bigr)
=
n\log\frac1\alpha+O(1).
\end{equation*}
Thus, \emph{(i)} and \emph{(ii)} are equivalent.

By Theorem~\ref{thm:herglotz_criterion}, condition~\emph{(ii)} is
equivalent to the logarithmic integrability condition
\eqref{eq:omega_extremal_condition} for the Herglotz measure
associated with $\widetilde g$. Hence, \emph{(ii)} and \emph{(iii)} are
equivalent, and therefore all three conditions are equivalent.

It remains to verify that the formulation is independent of the choice of the Riemann mapping. Let
$\widehat\psi:\Omega\to\mathbb{D}$ be another Riemann mapping and set
\begin{equation*}
T=\widehat\psi\circ\psi^{-1}:\mathbb{D}\to\mathbb{D}.
\end{equation*}
Then $T$ is a conformal automorphism of $\mathbb{D}$ and
\begin{equation*}
\widehat g
=
\widehat\psi\circ g\circ\widehat\psi^{-1}
=
T\circ\widetilde g\circ T^{-1}.
\end{equation*}
Thus, $\widehat g$ is conformally conjugate to $\widetilde g$. Since
every automorphism of $\mathbb{D}$ is a hyperbolic isometry,
\begin{equation*}
\rho_{\mathbb{D}}(T(\xi),T(\eta))
=
\rho_{\mathbb{D}}(\xi,\eta),
\quad \xi,\eta\in\mathbb{D}.
\end{equation*}
Consequently, $\widetilde g$ has an extremal hyperbolic rate if and
only if $\widehat g$ does. Moreover, their Denjoy--Wolff points and
angular multipliers correspond under $T$, and the multiplier is
unchanged.

Applying Theorem~\ref{thm:herglotz_criterion} to each conjugate map
shows that the corresponding logarithmic Herglotz condition holds
for one choice of the Riemann mapping if and only if it holds for any
other choice. Hence, the criterion, although expressed in terms of the
disk and its Herglotz measure, characterizes an intrinsic property of
the original dynamical system $(\Omega,g)$.
\end{proof}

\subsection{Extension to multiply connected domains}
\label{subsec:multiply_connected}

When $\Omega$ is multiply connected, the Riemann mapping theorem is no
longer available, and one cannot conjugate a self-map of $\Omega$ directly
to a self-map of the unit disk. The natural substitute is the universal
covering map
\begin{equation*}
\pi:\mathbb D\to\Omega,
\end{equation*}
which replaces conformal conjugation by a covering-space lift. We use the
standard correspondence between the hyperbolic metric of $\Omega$ and the
Poincar\'e metric of its universal covering disk; see, for example,
\cite[Sections 1.6, 1.7, and 3.3]{Abate19}.

Let $\Gamma$ denote the deck transformation group of $\pi$:
\begin{equation*}
\label{eq:deck_group}
\Gamma
:=
\{\gamma\in\operatorname{Aut}(\mathbb D):
\pi\circ\gamma=\pi\}.
\end{equation*}
The group $\Gamma$ acts freely and properly discontinuously on $\mathbb D$,
and $\Omega$ is conformally equivalent to the quotient
$\mathbb D/\Gamma$.

We recall that the limit set of $\Gamma$ is defined by
\begin{equation*}
\label{eq:limit_set}
\Lambda(\Gamma)
:=
\overline{\Gamma(z)}\cap\partial\mathbb D,
\quad z\in\mathbb D,
\end{equation*}
where the closure is taken in $\overline{\mathbb D}$. The resulting set is
independent of the choice of $z\in\mathbb D$. A point
$\tau\in\partial\mathbb D$ is called an \emph{ordinary point} of $\Gamma$
if
\begin{equation*}
\tau\notin\Lambda(\Gamma).
\end{equation*}
In particular, an ordinary point has a neighborhood on which the covering
map is locally one-to-one in a way that is uniform with respect to the
deck transformations.

The following lemma is the geometric ingredient needed to transfer the
extremal rate from the universal covering disk to the quotient domain.

\begin{lemma}
\label{lem:covering_distance}
Let $\pi:\mathbb D\to\Omega$ be a universal covering of a hyperbolic
domain $\Omega$, with deck transformation group $\Gamma$. Let
$\tau\in\partial\mathbb D\setminus\Lambda(\Gamma)$, and let
$\{z_n\}\subset\mathbb D$ converge to $\tau$. Suppose that
$z_n\to\tau$ non-tangentially. Fix $w\in\mathbb D$. Then there exists a
constant $C=C(w,\tau,\Gamma)>0$ such that
\begin{equation}
\label{eq:covering_distance_comparison}
\left|
\rho_\Omega\bigl(\pi(z_n),\pi(w)\bigr)
-
\rho_{\mathbb D}(z_n,w)
\right|
\le C
\end{equation}
for all sufficiently large $n$.
\end{lemma}

\begin{proof}
Since $\tau\notin\Lambda(\Gamma)$, there exists a Euclidean neighborhood
$U$ of $\tau$ in $\overline{\mathbb D}$ such that
\begin{equation*}
\label{eq:U_disjoint_translates}
U\cap\gamma(U)=\varnothing
\quad
\text{for every }\gamma\in\Gamma\setminus\{\operatorname{id}\}.
\end{equation*}
After shrinking $U$, we may also assume that $U\cap\mathbb D$ is
simply connected. Since $z_n\to\tau$, we have
$z_n\in U\cap\mathbb D$ for all sufficiently large $n$.

For $z_1,z_2\in\mathbb D$, the hyperbolic distance in the quotient is
given by
\begin{equation*}
\label{eq:quotient_distance}
\rho_\Omega\bigl(\pi(z_1),\pi(z_2)\bigr)
=
\inf_{\gamma\in\Gamma}
\rho_{\mathbb D}\bigl(z_1,\gamma(z_2)\bigr).
\end{equation*}
Consequently,
\begin{equation}
\label{eq:quotient_upper}
\rho_\Omega\bigl(\pi(z_n),\pi(w)\bigr)
\le
\rho_{\mathbb D}(z_n,w).
\end{equation}

It remains to obtain the reverse inequality up to a constant. Since
$\tau\notin\Lambda(\Gamma)$, the orbit $\Gamma w$ does not accumulate at
$\tau$. Thus, after shrinking $U$ if necessary,
\begin{equation*}
\label{eq:orbit_avoids_U}
\gamma(w)\notin U
\quad
\text{for every }\gamma\in\Gamma\setminus\{\operatorname{id}\}.
\end{equation*}
Let
\begin{equation*}
K:=\partial U\cap\mathbb D.
\end{equation*}
The set $K$ is compact in $\mathbb D$. Choose a point $q\in K$. For every
nontrivial $\gamma\in\Gamma$, every curve joining $z_n\in U$ to
$\gamma(w)\notin U$ must meet $K$. Hence,
\begin{equation*}
\rho_{\mathbb D}\bigl(z_n,\gamma(w)\bigr)
\ge
\rho_{\mathbb D}(z_n,K)
\ge
\rho_{\mathbb D}(z_n,q).
\end{equation*}
By the triangle inequality,
\begin{equation*}
\rho_{\mathbb D}(z_n,q)
\ge
\rho_{\mathbb D}(z_n,w)-\rho_{\mathbb D}(w,q).
\end{equation*}
Therefore,
\begin{equation}
\label{eq:nontrivial_lower}
\rho_{\mathbb D}\bigl(z_n,\gamma(w)\bigr)
\ge
\rho_{\mathbb D}(z_n,w)-\rho_{\mathbb D}(w,q)
\end{equation}
for every $\gamma\in\Gamma\setminus\{\operatorname{id}\}$ and all
sufficiently large $n$. Taking the infimum over $\gamma\in\Gamma$ and
combining \eqref{eq:quotient_upper} and \eqref{eq:nontrivial_lower} gives
\begin{equation*}
\rho_{\mathbb D}(z_n,w)-\rho_{\mathbb D}(w,q)
\le
\rho_\Omega\bigl(\pi(z_n),\pi(w)\bigr)
\le
\rho_{\mathbb D}(z_n,w),
\end{equation*}
which proves \eqref{eq:covering_distance_comparison}.
\end{proof}

The preceding lemma immediately yields the following extension of the
extremal-rate characterization.

\begin{theorem}
\label{thm:multiply_connected}
Let $\Omega\subsetneq\mathbb C$ be a hyperbolic domain,
$\pi:\mathbb D\to\Omega$ be a universal covering map, and
$g:\Omega\to\Omega$ be a holomorphic self-map without fixed points in
$\Omega$. Let
$\widetilde g:\mathbb D\to\mathbb D$
be a holomorphic lift satisfying
\begin{equation}
\label{eq:lift_relation}
\pi\circ\widetilde g=g\circ\pi.
\end{equation}
Assume that $\widetilde g$ is hyperbolic, with Denjoy--Wolff point
$\tau\in\partial\mathbb D$ and angular multiplier
$\alpha\in(0,1)$, and that
$\tau\notin\Lambda(\Gamma).$
Suppose, for some $z\in\mathbb D$, that
\begin{equation*}
\widetilde g^{\circ n}(z)\longrightarrow\tau
\end{equation*}
non-tangentially as $n\to\infty$. Then $g$ has an extremal hyperbolic
rate in $\Omega$ if and only if $\widetilde g$ has an extremal
hyperbolic rate in $\mathbb D$. Equivalently, $g$ has an extremal
hyperbolic rate if and only if the Herglotz measure $\sigma$ associated
with $\widetilde g$ satisfies
\begin{equation}
\label{eq:multiply_extremal_condition}
\int_{\partial\mathbb D\setminus\{\tau\}}
\log\frac{1}{|\zeta-\tau|}
\,d\sigma(\zeta)
<+\infty.
\end{equation}
\end{theorem}

\begin{proof}
Set
\begin{equation*}
z_n:=\widetilde g^{\circ n}(z).
\end{equation*}
Then, by the lifting relation \eqref{eq:lift_relation},
\begin{equation*}
\pi(z_n)=g^{\circ n}(\pi(z)).
\end{equation*}
Fix $w\in\mathbb D$ and put $w_0:=\pi(w)$. Since
$z_n\to\tau$ non-tangentially and $\tau\notin\Lambda(\Gamma)$,
Lemma~\ref{lem:covering_distance} gives a constant $C>0$ such that
\begin{equation}
\label{eq:orbit_distance_comparison}
\left|
\rho_\Omega\bigl(g^{\circ n}(\pi(z)),w_0\bigr)
-
\rho_{\mathbb D}(z_n,w)
\right|
\le C
\end{equation}
for all sufficiently large $n$.

Suppose first that $\widetilde g$ has an extremal hyperbolic rate. Then
\begin{equation*}
\rho_{\mathbb D}(z_n,w)
=
n\log\frac1\alpha+O(1).
\end{equation*}
It follows immediately from \eqref{eq:orbit_distance_comparison} that
\begin{equation*}
\rho_\Omega\bigl(g^{\circ n}(\pi(z)),w_0\bigr)
=
n\log\frac1\alpha+O(1).
\end{equation*}
The definition of extremal rate is independent of the initial point and
the base point, since the triangle inequality and the Schwarz--Pick
inequality give uniform-in-$n$ bounds when either point is changed.
Hence, $g$ has an extremal hyperbolic rate in $\Omega$.

Conversely, suppose that $g$ has an extremal hyperbolic rate. Then
\begin{equation*}
\rho_\Omega\bigl(g^{\circ n}(\pi(z)),w_0\bigr)
=
n\log\frac1\alpha+O(1).
\end{equation*}
Again by \eqref{eq:orbit_distance_comparison},
\begin{equation*}
\rho_{\mathbb D}(z_n,w)
=
n\log\frac1\alpha+O(1),
\end{equation*}
and therefore $\widetilde g$ has an extremal hyperbolic rate in
$\mathbb D$.

Finally, Theorem~\ref{thm:herglotz_criterion}, applied to the hyperbolic
lift $\widetilde g$, shows that its extremal hyperbolic rate is equivalent
to the logarithmic integrability condition
\eqref{eq:multiply_extremal_condition}. This proves the theorem.
\end{proof}

\begin{remark}
The ordinary-point assumption
$\tau\notin\Lambda(\Gamma)$ is essential for the preceding argument. It
ensures that the orbit approaching $\tau$ eventually lies in a region
where the different lifts of a fixed point remain uniformly separated
from the orbit. Consequently, passing from the covering disk to the
quotient changes the relevant hyperbolic distance by at most a bounded
additive amount. This is exactly the level of control required for the
extremal-rate condition.
\end{remark}

%-----------------------------------------------------
\section{Higher-Dimensional and Quasiconformal Settings}
\label{sec:higher}
%%=============================================================

The preceding sections established the extremal-rate theory for
holomorphic self-maps of hyperbolic domains in the complex plane. In the
simply connected case, the theory can be reduced to the unit disk, while
in the multiply connected case the universal covering map provides the
appropriate replacement. In one complex dimension, the Herglotz--Nevanlinna
representation additionally gives an explicit measure-theoretic
characterization of extremality.

We now consider two higher-dimensional and less rigid settings. First, we
study holomorphic self-maps of the unit ball
$\mathbb B^n\subset\mathbb C^n$. Second, we consider
$K$-quasiconformal self-maps of $\mathbb B^n$. In these settings there is
no direct analogue of the one-dimensional Herglotz representation with
the same explicit boundary measure, and we therefore work directly with
the intrinsic hyperbolic geometry.

The basic mechanism remains unchanged. The relevant boundary
linearization determines the exponential scale of the orbit, while the
hyperbolic metric converts this boundary scale into the asymptotic
distance
\begin{equation*}
n\log\frac1\alpha+O(1).
\end{equation*}
For the unit ball, the appropriate boundary geometry is described by
Kor\'anyi regions, and the role of the one-dimensional Koenigs
linearization is played by a Valiron linearizer.

\subsection{The unit ball}
\label{subsec:ball}

We use the standard Hermitian inner product
\begin{equation*}
\langle z,w\rangle
=
\sum_{j=1}^n z_j\overline{w_j}
\end{equation*}
on $\mathbb C^n$. For $\tau\in\partial\mathbb B^n$ and $M>1$, define the
Kor\'anyi region
\begin{equation}
\label{eq:koranyi_region}
K(\tau,M)
=
\left\{
z\in\mathbb B^n:
|1-\langle z,\tau\rangle|
<
M(1-|z|)
\right\}.
\end{equation}
Since
\begin{equation*}
1-|z|
\le
|1-\langle z,\tau\rangle|,
\end{equation*}
the condition $M>1$ guarantees that these regions are nonempty. Up to a
change in the constant $M$, \eqref{eq:koranyi_region} is equivalent to
the more customary formulation
\begin{equation*}
|1-\langle z,\tau\rangle|
<
M(1-|z|^2).
\end{equation*}
These regions describe the standard admissible, or Kor\'anyi,
non-tangential approach to $\tau$; see, for example,
\cite{Ko} and \cite[p.~170]{Rudin}.

We first record the boundary asymptotic of the hyperbolic metric that
will be used below.

\begin{lemma}
\label{lem:ball_asymp}
Let $\tau\in\partial\mathbb B^n$ and let $w\in\mathbb B^n$ be fixed.
For every $M>1$, there exists a constant
$C=C(M,w)>0$ such that
\begin{equation}
\label{eq:ball_asymp}
\left|
\rho_{\mathbb B^n}(z,w)
-
\log\frac{1}{|1-\langle z,\tau\rangle|}
\right|
\le C
\end{equation}
for all $z\in K(\tau,M)$ sufficiently close to $\tau$. In particular,
\begin{equation}
\label{eq:ball_asymp_o1}
\rho_{\mathbb B^n}(z,w)
=
\log\frac{1}{|1-\langle z,\tau\rangle|}
+
O(1)
\quad
(z\to\tau,\ z\in K(\tau,M)).
\end{equation}
\end{lemma}

\begin{proof}
By a unitary transformation, we may assume that
$\tau=e_1=(1,0,\ldots,0)$. Such a transformation is an isometry of the
hyperbolic metric, so it does not affect the assertion.

With the curvature normalization $-1$ used throughout the paper, the
hyperbolic distance from the origin is given by (see \eqref{hyperb-unit ball} with $w=0$)
\begin{equation*}
\label{eq:ball_origin_distance}
\sinh^2
\frac{\rho_{\mathbb B^n}(z,0)}{2}
=
\frac{|z|^2}{1-|z|^2}.
\end{equation*}
Consequently,
\begin{align*}
\rho_{\mathbb B^n}(z,0)
&=
2\operatorname{arsinh}
\frac{|z|}{\sqrt{1-|z|^2}}
=
2\log
\left(
\frac{1+|z|}{\sqrt{1-|z|^2}}
\right)
\\
&=
\log\frac{1}{1-|z|^2}
+
2\log(1+|z|).
\label{eq:ball_origin_asymp}
\end{align*}
Hence,
\begin{equation*}
\label{eq:ball_origin_o1}
\rho_{\mathbb B^n}(z,0)
=
\log\frac{1}{1-|z|^2}
+
O(1),
\quad
(|z|\to1),
\end{equation*}
where the bounded term is in fact between $0$ and $\log4$.

For arbitrary fixed $w\in\mathbb B^n$, the triangle inequality gives
\begin{equation*}
\left|
\rho_{\mathbb B^n}(z,w)
-
\rho_{\mathbb B^n}(z,0)
\right|
\le
\rho_{\mathbb B^n}(0,w).
\end{equation*}
Thus,
\begin{equation}
\label{eq:ball_general_radial}
\rho_{\mathbb B^n}(z,w)
=
\log\frac{1}{1-|z|^2}
+
O(1),
\end{equation}
where the bounded term depends only on $w$.

It remains to compare $1-|z|^2$ with
$|1-\langle z,\tau\rangle|$. Since $z\in K(e_1,M)$,
\begin{equation*}
\label{eq:ball_upper_comparison}
|1-z_1|
<
M(1-|z|)
\le
M(1-|z|^2).
\end{equation*}
On the other hand,
\begin{equation*}
1-|z|^2
\le
1-|z_1|^2
=
(1-|z_1|)(1+|z_1|)
\le
2(1-|z_1|)
\le
2|1-z_1|.
\end{equation*}
Therefore,
\begin{equation*}
\label{eq:ball_comparability}
\frac{1}{M}|1-z_1|
\le
1-|z|^2
\le
2|1-z_1|.
\end{equation*}
Since $\tau=e_1$,
$|1-z_1|=|1-\langle z,\tau\rangle|$, and hence,
\begin{equation*}
\log\frac{1}{1-|z|^2}
=
\log\frac{1}{|1-\langle z,\tau\rangle|}
+
O(1),
\end{equation*}
where the bounded term depends only on $M$. Combining this with
\eqref{eq:ball_general_radial} proves
\eqref{eq:ball_asymp} and \eqref{eq:ball_asymp_o1}. This completes the proof.
\end{proof}

The preceding lemma is the higher-dimensional analogue of the boundary
asymptotic for the hyperbolic metric in the disk. We can now formulate a
sufficient criterion for extremal hyperbolic growth.

\begin{theorem}
\label{thm:ball}
Let
$g:\mathbb B^n\to\mathbb B^n$ be a hyperbolic holomorphic self-map with
Denjoy--Wolff point $\tau\in\partial\mathbb B^n$ and boundary dilation
coefficient $\alpha\in(0,1)$. Suppose that:

\begin{enumerate}
\item[\emph{(a)}]
for every $z_0\in\mathbb B^n$, the forward orbit
$z_n=g^{\circ n}(z_0)$ converges to $\tau$ non-tangentially, in the sense
that there exists $M>1$, depending possibly on $z_0$, such that
\begin{equation*}
z_n\in K(\tau,M)
\end{equation*}
for all sufficiently large $n$;

\item[\emph{(b)}]
there exists a holomorphic Valiron linearizer
$h:\mathbb B^n\to\mathbb H$ satisfying
\begin{equation}
\label{eq:ball_valiron}
h(g(z))
=
\alpha^{-1}h(z),
\quad z\in\mathbb B^n,
\end{equation}
and, for every $M>1$,
\begin{equation}
\label{eq:ball_nondegenerate}
|h(z)|
\asymp
\frac{1}{|1-\langle z,\tau\rangle|},
\quad
(z\to\tau,\ z\in K(\tau,M)),
\end{equation}
where the comparison constants may depend on $M$.
\end{enumerate}

Then, for every $z_0,w\in\mathbb B^n$,
\begin{equation*}
\label{eq:ball_extremal_rate}
\rho_{\mathbb B^n}
\bigl(g^{\circ n}(z_0),w\bigr)
=
n\log\frac1\alpha
+
O(1),
\quad
(n\to\infty).
\end{equation*}
In particular, $g$ has an extremal hyperbolic rate.
\end{theorem}

\begin{proof}
Fix $z_0,w\in\mathbb B^n$ and write
\begin{equation*}
z_n=g^{\circ n}(z_0).
\end{equation*}
By the Denjoy--Wolff theorem for the ball,
$z_n\to\tau$. By assumption~\emph{(a)}, the orbit is eventually contained
in a fixed Kor\'anyi region $K(\tau,M)$.

Iterating the functional equation \eqref{eq:ball_valiron} gives
\begin{equation*}
\label{eq:ball_valiron_iterates}
h(z_n)
=
\alpha^{-n}h(z_0).
\end{equation*}
Therefore,
\begin{equation}
\label{eq:ball_log_h}
\log|h(z_n)|
=
n\log\frac1\alpha
+
\log|h(z_0)|.
\end{equation}
Since $z_n\to\tau$ within $K(\tau,M)$, assumption~\emph{(b)} yields
\begin{equation*}
\log|h(z_n)|
=
\log\frac{1}{|1-\langle z_n,\tau\rangle|}
+
O(1),
\end{equation*}
where the $O(1)$ term is independent of $n$. Combining this with
\eqref{eq:ball_log_h}, and absorbing the fixed quantity
$\log|h(z_0)|$ into the bounded term, gives
\begin{equation}
\label{eq:ball_boundary_growth}
\log\frac{1}{|1-\langle z_n,\tau\rangle|}
=
n\log\frac1\alpha
+
O(1).
\end{equation}

Finally, Lemma~\ref{lem:ball_asymp} gives
\begin{equation*}
\rho_{\mathbb B^n}(z_n,w)
=
\log\frac{1}{|1-\langle z_n,\tau\rangle|}
+
O(1).
\end{equation*}
Substituting \eqref{eq:ball_boundary_growth} into this relation yields
\begin{equation*}
\rho_{\mathbb B^n}
\bigl(g^{\circ n}(z_0),w\bigr)
=
n\log\frac1\alpha
+
O(1),
\end{equation*}
as claimed.
\end{proof}

\begin{remark}
The proof of Theorem~\ref{thm:ball} uses only the existence of a
Valiron linearizer satisfying \eqref{eq:ball_valiron} together with the
non-degeneracy condition \eqref{eq:ball_nondegenerate}. The
higher-dimensional Julia--Wolff--Carath\'eodory theorem identifies the
boundary dilation coefficient $\alpha$ and provides the corresponding
first-order boundary behavior, but no additional Julia--Wolff--Carath\'eodory
estimate is needed once the Valiron functional equation and the
non-degeneracy assumption are available.
\end{remark}
%%=============================================================

\subsection{Quasiconformal self-maps}
\label{subsec:qc}

We now consider $K$-quasiconformal self-maps of the unit ball.
Recall that a homeomorphism
$f:\mathbb{B}^n\to\mathbb{B}^n$ is called $K$-quasiconformal,
$K\geq1$, if, after identifying
$\mathbb{C}^n$ with $\mathbb{R}^{2n}$, one has
$f\in W_{\mathrm{loc}}^{1,2n}(\mathbb{B}^n)$ and
\begin{equation*}
\|Df(z)\|^{2n}\leq K J_f(z)
\end{equation*}
for almost every $z\in\mathbb{B}^n$, where $\|Df(z)\|$ denotes the
operator norm of the real differential and $J_f(z)$ is the Jacobian
determinant. We assume, as usual, that $f$ is orientation preserving.
See, for example, \cite{HKV, Vuorinen}.

Unlike the holomorphic case, we do not have a holomorphic
linearization or a boundary angular derivative available in general
for quasiconformal self-maps. Instead, the following criterion
formulates directly the boundary asymptotics of the complex normal
distance that governs the hyperbolic metric in a Kor\'anyi approach
region.

\begin{theorem}
\label{thm:qc}
Let $f:\mathbb{B}^n\to\mathbb{B}^n$ be a $K$-quasiconformal
self-map with no fixed point in $\mathbb{B}^n$. Suppose that there
exist a point $\tau\in\partial\mathbb{B}^n$, a number
$\alpha\in(0,1)$, and $\beta>0$ such that:

\begin{enumerate}
\item[\emph{(i)}]
for every $z\in\mathbb{B}^n$, the forward orbit
$f^{\circ n}(z)$ converges to $\tau$ non-tangentially, in the sense
that it eventually remains in some Kor\'anyi region
$K(\tau,M)$;

\item[\emph{(ii)}]
as $z\to\tau$ non-tangentially,
\begin{equation}
\label{eq:qc_normal_asymptotic}
\frac{|1-\langle f(z),\tau\rangle|}
{|1-\langle z,\tau\rangle|}
=
\alpha\left(
1+O\!\left(|1-\langle z,\tau\rangle|^\beta\right)
\right).
\end{equation}
\end{enumerate}

Then, for every $z_0,w\in\mathbb{B}^n$,
\begin{equation}
\label{eq:qc_extremal_rate}
\rho_{\mathbb{B}^n}
\bigl(f^{\circ n}(z_0),w\bigr)
=
n\log\frac{1}{\alpha}+O(1),
\quad n\to\infty.
\end{equation}
\end{theorem}

\begin{proof}
Fix $z_0,w\in\mathbb{B}^n$ and set
\begin{equation*}
z_n=f^{\circ n}(z_0),
\quad
d_n:=|1-\langle z_n,\tau\rangle|.
\end{equation*}
By assumption~\emph{(i)}, $z_n\to\tau$ non-tangentially and hence,
$d_n\to0$. In particular, there exist $M>1$ and $N_0\in\mathbb{N}$
such that
\begin{equation*}
z_n\in K(\tau,M)
\quad\text{for all }n\geq N_0.
\end{equation*}

Applying~\eqref{eq:qc_normal_asymptotic} to $z=z_n$ gives
\begin{equation}
\label{eq:qc_recurrence}
d_{n+1}
=
\alpha d_n(1+\eta_n),
\quad
\eta_n=O(d_n^\beta),
\end{equation}
as $n\to\infty$.

Choose $\varepsilon>0$ such that
$\alpha+\varepsilon<1$. Since $\eta_n\to0$, there exists
$N\geq N_0$ such that
\begin{equation*}
|1+\eta_n|\leq 1+\frac{\varepsilon}{\alpha},
\quad n\geq N.
\end{equation*}
Because the left-hand side of~\eqref{eq:qc_recurrence} is positive,
we may, after increasing $N$ if necessary, assume that
$1+\eta_n>0$ and
\begin{equation*}
d_{n+1}\leq(\alpha+\varepsilon)d_n,
\quad n\geq N.
\end{equation*}
Consequently,
\begin{equation*}
d_n
\leq
d_N(\alpha+\varepsilon)^{n-N},
\quad n\geq N,
\end{equation*}
and hence,
\begin{equation*}
\label{eq:qc_dn_decay}
d_n=O\bigl((\alpha+\varepsilon)^n\bigr).
\end{equation*}
It follows from~\eqref{eq:qc_recurrence} that
\begin{equation*}
\label{eq:qc_eta_decay}
\eta_n
=
O\bigl((\alpha+\varepsilon)^{n\beta}\bigr),
\end{equation*}
so that
\begin{equation}
\label{eq:qc_eta_summable}
\sum_{n=N}^{\infty}|\eta_n|<\infty.
\end{equation}

Iterating~\eqref{eq:qc_recurrence}, we obtain, for $n>N$,
\begin{equation*}
d_n
=
d_N\alpha^{n-N}
\prod_{k=N}^{n-1}(1+\eta_k).
\end{equation*}
Since $\eta_k\to0$ and~\eqref{eq:qc_eta_summable} holds, the series
\begin{equation*}
\sum_{k=N}^{\infty}\log(1+\eta_k)
\end{equation*}
converges. Taking logarithms therefore gives
\begin{equation*}
\log d_n
=
(n-N)\log\alpha
+
\log d_N
+
\sum_{k=N}^{n-1}\log(1+\eta_k)
=
n\log\alpha+O(1).
\end{equation*}
Equivalently,
\begin{equation}
\label{eq:qc_boundary_rate}
\log\frac{1}{d_n}
=
n\log\frac{1}{\alpha}+O(1).
\end{equation}

Finally, since $z_n$ eventually remains in the Kor\'anyi region
$K(\tau,M)$, Lemma~\ref{lem:ball_asymp} yields
\begin{equation*}
\rho_{\mathbb{B}^n}(z_n,w)
=
\log\frac{1}{|1-\langle z_n,\tau\rangle|}
+
O(1)
=
\log\frac{1}{d_n}+O(1),
\end{equation*}
where the error term is independent of $n$. Combining this with
\eqref{eq:qc_boundary_rate} gives
\begin{equation*}
\rho_{\mathbb{B}^n}(f^{\circ n}(z_0),w)
=
n\log\frac{1}{\alpha}+O(1),
\end{equation*}
as claimed.
\end{proof}

\begin{remark}
\label{rem:qc}
The role of assumption~\emph{(ii)} is to provide, in the
quasiconformal setting, the boundary asymptotic that in the
holomorphic setting is supplied by the Julia--Wolff--Carath\'eodory
theorem. More precisely,~\eqref{eq:qc_normal_asymptotic} states that
the complex normal distance to the boundary point $\tau$ is
asymptotically contracted by the factor $\alpha$ under $f$, with a
H\"older-type error. Along a forward orbit this error is summable,
which yields the bounded remainder in~\eqref{eq:qc_extremal_rate}.

Thus, the mechanism behind the extremal rate is the same as in the
holomorphic case: the relevant boundary scale decreases
geometrically with ratio $\alpha$, while the accumulated deviation
from this linear model remains bounded. The difference is that, for
a general quasiconformal self-map, this boundary behavior must be
imposed as an explicit hypothesis rather than obtained from a
holomorphic angular-derivative theory.
\end{remark}

%%%%%%%%%%%%%%%%%%%%%%%%%%%%%%%%%%%%%%%%%%%%%%%%%%%%%%%%%%%%55
\subsection{Quasiconformal maps on multiply connected domains}
\label{subsec:qc_multiply}

Combining Theorems~\ref{thm:multiply_connected} and~\ref{thm:qc},
we obtain the following extension of the extremal rate theory to
quasiconformal self-maps of multiply connected hyperbolic domains.

\begin{theorem}
\label{thm:qc_multiply}
Let $\Omega\subsetneq\mathbb{C}$ be a hyperbolic domain with finitely
many boundary components, and let
$F:\Omega\to\Omega$ be a $K$-quasiconformal homeomorphism with no fixed
point in $\Omega$. Let
$\pi:\mathbb{D}\to\Omega$ be a universal covering map, with deck
transformation group $\Gamma$, and let
$\widetilde F:\mathbb{D}\to\mathbb{D}$ be a $K$-quasiconformal lift
satisfying
\begin{equation*}
\pi\circ\widetilde F
=
F\circ\pi.
\end{equation*}
Assume that $\tau\in\partial\mathbb{D}$ is an ordinary point of
$\Gamma$, that is, $\tau\notin\Lambda(\Gamma)$, and suppose that:

\begin{enumerate}
\item[\emph{(i)}]
for every $z\in\mathbb{D}$, the forward orbit
$\widetilde F^{\circ n}(z)$ converges to $\tau$ non-tangentially,
that is, it eventually remains in some Stolz region
$S(\tau,M)$, where $M$ may depend on $z$;

\item[\emph{(ii)}]
there exist $\alpha\in(0,1)$ and $\beta>0$ such that
\begin{equation*}
\label{eq:qc_multiply_boundary}
\frac{|1-\overline{\tau}\,\widetilde F(z)|}
{|1-\overline{\tau}\,z|}
=
\alpha
\left(
1+
O\!\left(
|1-\overline{\tau}\,z|^\beta
\right)
\right)
\end{equation*}
as $z\to\tau$ non-tangentially within every Stolz region.
\end{enumerate}

Then $F$ has an extremal hyperbolic rate in $\Omega$ if and only if
$\widetilde F$ has an extremal hyperbolic rate in $\mathbb{D}$.
In either case, the extremal rate is
\begin{equation*}
n\log\frac{1}{\alpha}+O(1).
\end{equation*}
\end{theorem}

\begin{proof}
Fix arbitrary points $z,w\in\mathbb{D}$ and set
\begin{equation*}
z_n=\widetilde F^{\circ n}(z),
\quad
z'=\pi(z),
\quad
w'=\pi(w).
\end{equation*}
Iterating the lifting identity gives
\begin{equation*}
\label{eq:qc_multiply_lift}
\pi\circ\widetilde F^{\circ n}
=
F^{\circ n}\circ\pi,
\quad n\geq0,
\end{equation*}
and hence,
\begin{equation*}
\pi(z_n)=F^{\circ n}(z').
\end{equation*}

By assumption~\emph{(i)}, the orbit $\{z_n\}$ eventually remains in
a Stolz region $S(\tau,M)$. Since $\tau$ is an ordinary point of the
deck transformation group, the covering-space argument used in the
proof of Theorem~\ref{thm:multiply_connected} applies. In particular,
for every fixed $z,w\in\mathbb{D}$ there exists $N=N(z,w)$ such that
\begin{equation}
\label{eq:qc_multiply_distance}
\rho_\Omega\bigl(F^{\circ n}(z'),w'\bigr)
=
\rho_{\mathbb{D}}(z_n,w)
=
\rho_{\mathbb{D}}
\bigl(\widetilde F^{\circ n}(z),w\bigr)
\end{equation}
for all $n\geq N$.

Indeed, the first equality follows from the fact that, once $z_n$
lies in a sufficiently small neighborhood of the ordinary point
$\tau$, the minimizing lift of the hyperbolic distance from
$F^{\circ n}(z')=\pi(z_n)$ to $w'=\pi(w)$ is represented by the fixed
lift $w$, rather than by a nontrivial deck translate of $w$. This is
precisely the covering-space comparison established in
Theorem~\ref{thm:multiply_connected}.

Now assumptions~\emph{(i)} and~\emph{(ii)} are exactly the hypotheses
of Theorem~\ref{thm:qc} for the lifted map $\widetilde F$. Therefore,
for every $z,w\in\mathbb{D}$,
\begin{equation}
\label{eq:qc_multiply_lift_rate}
\rho_{\mathbb{D}}
\bigl(\widetilde F^{\circ n}(z),w\bigr)
=
n\log\frac{1}{\alpha}+O(1),
\quad n\to\infty.
\end{equation}
Combining~\eqref{eq:qc_multiply_distance} and
\eqref{eq:qc_multiply_lift_rate} gives
\begin{equation*}
\rho_\Omega
\bigl(F^{\circ n}(z'),w'\bigr)
=
n\log\frac{1}{\alpha}+O(1).
\end{equation*}
Since $z',w'\in\Omega$ were arbitrary, $F$ has an extremal
hyperbolic rate in $\Omega$.

Conversely, suppose that $F$ has an extremal hyperbolic rate in
$\Omega$. For arbitrary $z,w\in\mathbb{D}$, the eventual identity
\eqref{eq:qc_multiply_distance} gives
\begin{equation*}
\rho_{\mathbb{D}}
\bigl(\widetilde F^{\circ n}(z),w\bigr)
=
\rho_\Omega
\bigl(F^{\circ n}(\pi(z)),\pi(w)\bigr)
=
n\log\frac{1}{\alpha}+O(1).
\end{equation*}
Thus, $\widetilde F$ has an extremal hyperbolic rate in $\mathbb{D}$.
This proves the equivalence.
\end{proof}

%=============================
%\subsection*{Funding}
%No external funding was received for this research.

%\subsection*{Acknowledgments}

%\subsection*{Funding}

%\subsection*{Data Availability}
\noindent
\textbf{Conflict of Interest.}
The author declares no conflicts of interest.
%=====================================================================
%\noindent
%\textbf{Use of generative AI.} In preparing this work, the author used Grammarly (Free AI Writing Assistance) to help improve the English language and conciseness of the text. After using this tool, the author reviewed and edited the content as needed and takes full responsibility for the final manuscript.


\begin{thebibliography}{99}

\bibitem{Abate19}
M. Abate, Holomorphic Dynamics on Hyperbolic Riemann Surfaces. Berlin, Boston: De Gruyter, 2023.


\bibitem{Abate}
M.~Abate,
Iteration Theory of Holomorphic Maps on Taut Manifolds.
Res. Lecture Notes Math. Complex Anal. Geom.,
Mediterranean Press, Rende, 1989.

\bibitem{Bak-Pom}
I.N. Baker, and Ch. Pommerenke,
\textit{On the iteration of analytic functions in a halfplane. II.}
J. London Math. Soc. (2) \textbf{20} (1979), no. 2, 255--258.

\bibitem{B83}
A.~F.~Beardon,
The Geometry of Discrete Groups.
Grad. Texts in Math., vol.~91,
Springer-Verlag, New York, 1995.

\bibitem{BM}
A.~F.~Beardon and D.~Minda,
The hyperbolic metric and geometric function theory,
in Quasiconformal Mappings and Their Applications,
Narosa Publishing House, New Delhi, 2007, pp.~9--56.

\bibitem{bracci2020}
F.~Bracci, M.~D.~Contreras, and S.~D\'iaz-Madrigal,
Continuous Semigroups of Holomorphic Self-Maps of the Unit Disc.
Springer Monogr. Math.,
Springer, Cham, 2020.

\bibitem{Bracci2010}
F.~Bracci, G.~Gentili, and P.~Poggi-Corradini,
\textit{Valiron's construction in higher dimension}. Rev. Mat. Iberoamericana \textbf{26} (2010), no.~1, 57--76.

\bibitem{Cau}
W.~Cauer,
\textit{The Poisson integral for functions with positive real part}. Bull. Amer. Math. Soc. \textbf{38} (1932), no.~10, 713--717.

\bibitem{Cowen81}
C.~C. Cowen,
\emph{Iteration and the solution of functional equations for functions analytic in the unit disk},
Trans. Amer. Math. Soc. \textbf{265} (1981), no.~1, 69--95.

\bibitem{CowenMacCluer}
C.~C.~Cowen and B.~D.~MacCluer,
Composition Operators on Spaces of Analytic Functions.
Stud. Adv. Math.,
CRC Press, Boca Raton, FL, 1995.

\bibitem{CDG2025parabolic}
M.~D. Contreras, S.~D\'iaz-Madrigal, and P.~Gumenyuk,
\emph{Simultaneous linearization and centralizers of parabolic self-maps I: zero hyperbolic step},
Anal. Math. Phys. \textbf{15} (2025), Paper No.~137.

\bibitem{CDG2024}
M.~D. Contreras, S.~D\'iaz-Madrigal, and P.~Gumenyuk,
\textit{Centralizers of non-elliptic univalent self-maps and the embeddability problem in the unit disc}. arXiv:2311.04134v3 [math.CV]

\bibitem{cruz}
F.~J.~Cruz-Zamorano and K.~Zarvalis,
\textit{Extremal rate of convergence in discrete hyperbolic and parabolic dynamics}.
Preprint, 2025, \url{https://arxiv.org/abs/2510.12501v1}.

\bibitem{cruz26}
F.~J.~Cruz-Zamorano and K.~Zarvalis, \textit{Extremal rate of convergence in continuous dynamics}. J. Lond. Math. Soc. (2) \textbf{114} (2026), no. 1, Paper No. e70627.

\bibitem{elliott}
S.~Elliott and M.~T.~Jury,
\textit{Composition operators on Hardy spaces of a half-plane}. Bull. London Math. Soc. \textbf{44} (2012), no.~3, 489--495.

\bibitem{HKV}
P.~Hariri, R.~Kl\'en, and M.~Vuorinen, Conformally Invariant Metrics and Quasiconformal Mappings.
Springer Monogr. Math.,
Springer, Cham, 2020.

\bibitem{K}
G.~Koenigs,
\textit{Recherches sur les int\'egrales de certaines \'equations fonctionnelles}. Ann. Sci. \'Ecole Norm. Sup. \textbf{3} (1884), no.~1, 3--41.

\bibitem{Ko}
A.~Kor\'anyi,
\textit{Harmonic functions on Hermitian hyperbolic space}. Trans. Amer. Math. Soc. \textbf{135} (1969), 507--516.

\bibitem{Mac}
B. D. MacCluer,
\textit{Iterates of holomorphic self-maps of the unit ball in $\mathbb{C}^N$}.
Michigan Math. J. \textbf{30} (1983), no. 1, 97--106.

\bibitem{moor}
J.~Moorhouse,
\textit{Compact differences of composition operators}. J. Funct. Anal. \textbf{219} (2005), 70--92.

\bibitem{Nev}
R.~Nevanlinna,
\textit{Asymptotische Entwicklungen beschr\"ankter Funktionen und
das Stieltjessche Momentenproblem}. Ann. Acad. Sci. Fenn. Ser. A \textbf{18} (1922), no.~5, 1--53.

\bibitem{Pom79}
Ch. Pommerenke, \textit{On the iteration of analytic functions in a halfplane}. J. London Math. Soc.(2) \textbf{19} (1979), no. 3, 439--447.


\bibitem{Rudin}
W.~Rudin,
Function Theory in the Unit Ball of $\mathbb{C}^n$.
Grundlehren Math. Wiss., vol.~241,
Springer-Verlag, New York, 1980.

\bibitem{Shapiro}
J.~H.~Shapiro,
Composition Operators and Classical Function Theory.
Universitext,
Springer-Verlag, New York, 1993.

\bibitem{Shoikhet}
D.~Shoikhet,
Semigroups in Geometrical Function Theory.
Kluwer Academic Publishers, Dordrecht, 2001.

\bibitem{Vuorinen}
M.~Vuorinen,
Conformal Geometry and Quasiregular Mappings.
Lecture Notes in Math., vol.~1319,
Springer-Verlag, Berlin, 1988.


\end{thebibliography}
\end{document}